\documentclass[10pt]{article}
\usepackage[english]{babel}
\usepackage[utf8]{inputenc}
\usepackage[T1]{fontenc}
\usepackage{amsfonts,amsmath,amssymb,amsthm}
\usepackage{enumerate}
\usepackage{caption}
\usepackage{hyperref}
\usepackage{multirow}
\usepackage{manfnt}
\usepackage[ruled,vlined]{algorithm2e}
\usepackage{graphicx}
\usepackage{float}
\usepackage{enumitem}

\newtheorem{Theorem}{Theorem}

\newtheorem{Lemma}[Theorem]{Lemma}
\newtheorem{Proposition}[Theorem]{Proposition}
\newtheorem{Corollary}[Theorem]{Corollary}
\newtheorem{Problem}{Problem}

\def\vect#1{\mathbf{#1}}

\title{Bounding the number of $(\sigma,\rho)$-dominating sets in trees, forests and graphs of bounded pathwidth}
\author{Matthieu Rosenfeld}
\begin{document}
\maketitle
\begin{abstract}
The notion of $(\sigma,\rho)$-dominating set generalizes many notions including dominating set, induced matching, perfect codes or independent sets.
Bounds on the maximal number of such (maximal, minimal) sets were established for different $\sigma$ and $\rho$ and different classes of graphs.
In particular, Rote showed that the number of minimal dominating sets in trees of order $n$ is at most $95^{\frac{n}{13}}$ and 
Golovach et Al. computed the asymptotic of the number of $(\sigma,\rho)$-dominating sets in paths for all $\sigma$ and $\rho$.

Here, we propose a method to compute bounds on the number of $(\sigma,\rho)$-dominating sets in graphs or bounded pathwidth, trees and forests,
under the conditions that $\sigma$ and $\rho$ are finite unions of (possibly infinite) arithmetic progressions.
It seems that this method shouldn't always work, 
but in practice we are able to give many sharp bounds by direct application of the method.
Moreover, in the case of graphs of bounded pathwidth, we deduce the existence of an algorithm that can output abritrarily good approximations of the growth rate.
\end{abstract}
\section{Introduction}
The notion of $(\sigma, \rho)$-domination was introduced by Telle as a generalization of domination-type problems \cite{Telle}. 
Given two sets $\sigma$  and $\rho$ of non-negative integers, a vertex subset $D\subseteq V$ is a 
$(\sigma, \rho)$-dominating set of a given graph $G=(V,E)$, if $|N(v)\cap D|\in\sigma$ for all $v\in D$ 
and $|N(v)\cap D|\in\rho$ for all $v\in V\setminus D$. 
This framework can be used to express several well-known graph problems, like Dominating Set
$(\sigma =\mathbb{N}, \rho= \mathbb{N}^+ )$, Independent Set $(\sigma =\{0\}, \rho= \mathbb{N})$, 
Independent Dominating Set $(\sigma =\{0\}, \rho= \mathbb{N}^+ )$, Induced Matching $(\sigma =\{1\}, \rho= \mathbb{N})$, etc.

Bounding the number of such sets is of certain interest for algorithmic.
For instance, the celebrated bound by Moon and Moser of $3^{\frac{n}{3}}$
on the maximal number of maximal independent sets in a graph of order $n$ 
was used by Lawler to give an algorithm computing an optimal coloring of a graph in 
$O^*((1+3^{\frac{1}{3}})^n)$ \cite{Moon1965,Lawler1976ANO}.
Recently Rote showed that the maximal number of minimal dominating sets of trees of order $n$ is $95^{\frac{n}{13}}$ \cite{r-mnmds-19,arxivRote}.
On the other hand, Golovach et Al. gave asymptotics on the number of (all, minimal, maximal) $(\sigma,\rho)$-dominating sets in paths 
(remark that in the case of path there are only finitely many interesting $\sigma$ and $\rho$) \cite{sigmarhopath}.
Here, we generalize the approach of these two articles and we describe a computer assisted technique that computes sharp bounds on the number of 
(all, minimal, maximal) $(\sigma,\rho)$-dominating sets. 
We require $\sigma$ and $\rho$ to be finite unions of arithmetic progressions, which is the case for many classical problems (in particular, if $\sigma$ and $\rho$ are finite or co-finite).
In theory this method might not work all the time, but in practice we can use it to give many bounds in the last section.

We start by giving some useful and non standard definitions in Section \ref{secdef}.
In Section \ref{secalgo}, we give a general algorithm to count the number of (all, minimal, maximal) $(\sigma,\rho)$-dominating sets in graphs of bounded pathwidth, trees, and forests.
Finally, in Section \ref{secbounds}, we use the fact that the operations of the algorithm are (multi)linear operations to compute the bounds.
Then using a C++ implementation of the described technique, we can give some examples of bounds that we were able to compute.

We allow ourself to omit a lot of proof details, most of them being rather trivial but painful to write and to read.

\section{Definitions and notations}\label{secdef}
We denote the set of all non-negative integers by $\mathbb{N}$ and the set of positive integers by $\mathbb{N}^+$.

For any given graph $G=(V,E)$ and any vertex $v\in V$, let $N_G(v)$ be the neighbors of $v$ in $G$.
We will omit the $G$ and write $N(v)$ when $G$ is clear in the context.
The \emph{order} of a graph is the number of vertices of the graph.

For any set $X\subseteq\mathbb{N}$, we denote by $1_X:\mathbb{N}\mapsto\{0,1\}$ the indicator function of $X$, that is:
$$1_X(n)=\left\{
  \begin{array}{ll}
      1 & \text{ if }n\in X\\
      0 & \text{ otherwise } \\
    \end{array}
  \right.$$
  
For any positive integers $a$ and $b$, we denote by $a\mod b$ the remainder of the euclidean division of $a$ by $b$.
  
Let $\tau:\mathbb{N}^3\mapsto\mathbb{N}$ be the function such that for all $p,n\in\mathbb{N}$ and $q\in\mathbb{N}^+$:
$$\tau(p,q,n)=\left\{
  \begin{array}{ll}
      n & \text{ if } n <p\\
      ((n-p)\mod q)+p & \text{ otherwise } \\
    \end{array}
  \right.$$
Remark that for any $m,n\in \mathbb{N}$, $\tau(p,q,n+m)=\tau(p,q,\tau(p,q,n)+m)$.
  
For any vector $\vect{v}\in\mathbb{R}^n$, $\vect{v}_{i-1}$ is the $i$th coordinate of $\vect{v}$. 
For any two vectors $\vect{v},\vect{v'}\in\mathbb{R}^n$, $\vect{v} \le \vect{v'}$ means that for all $i\in\{0,\ldots, n-1\}$, $\vect{v}_i \le \vect{v}_i'$.

Let $F$ be a family of sets over a ground set $T$, then a set $X$ is 1-minimal if for all $x\in X$, $X\setminus \{x\} \not\in F$ and it is 1-maximal if for any $x\in T\setminus X$, $X\cup\{x\}\not\in F$. 

\subsection{Recognizable sets of integers}
We say that a set of integers $X\subseteq\mathbb{N}$ is \emph{recognizable} if the language $\{a^n|n\in X\}$ is recognizable.
It is a standard result that $X\subseteq\mathbb{N}$ is recognizable if and only if it is the finite union of arithmetic progressions (and this is the last mention of ``arithmetic progressions'' in this article).

By standard manipulations of automatons over one letter, we can get the following result:
\begin{Lemma}\label{functionautomaton}
 For any recognizable set of integers $S$, there exist $p_S,q_S\in \mathbb{N}$ such that
for all $n$, $1_S(n)=1_S(\tau(n,p_S,q_S))$.
\end{Lemma}
Remark that this is equivalent to say that $1_S$ is ultimately periodic.

We will also need a slightly stronger version:
\begin{Lemma}\label{recog}
Let $\sigma,\rho\subseteq\mathbb{N}$ be recognizable sets of integers.
Then there are $p,q\in \mathbb{N}$ such that:
\begin{itemize}
 \item for all $n$, $(1_\sigma(n),1_\rho(n))=(1_\sigma(\tau(n,p,q)),1_\rho(\tau(n,p,q)))$,
 \item $(1_\sigma(p-1),1_\rho(p-1))=(1_\sigma(p+q-1),1_\rho(p+q-1))$.
\end{itemize}
\end{Lemma}
This can be deduced from the previous Lemma and by taking $q= \gcd(q_\rho,q_\sigma)$ and $p=\max(p_\rho,p_\sigma)+1.$

\subsection{Pathwidth, trees and forests}
We give the definitions of trees, forests, and graphs of pathwidth $k$.
These definitions are not the standard definitions,
but it is not hard to see that they are equivalent and it is left to the interested reader to check it.

A \emph{$k$-distinguished graph} is a pair $(G(V,E),S,<_S)$ where $G(V,E)$ is a graph, $S$ is a subset of $V$ with $|S|=k$ and $(S\times S)\cap E=\emptyset$
and $<_S$ is a strict total order on $S$. We say that the vertices of $S$ are the \emph{distinguished vertices} of the $k$-distinguished graph.

We say that \emph{a graph $G'(V',E')$ is obtained by completing a $k$-distinguished graph $(G(V,E),S,<_S)$}, if $E'=E\cup (E'\cap (S'\times S'))$.
In other words, one \emph{completes} a $k$-distinguished graph by adding some edges between the elements of $S$.

\paragraph{Pathwidth}
We say that \emph{a $k$-distinguished graph $(G'(V',E'),S',<_{S'})$ is obtained by extending $(G(V,E),S,<_S)$}, 
if there exists $n\in V'$ and $o\in S\cup\{n\}$ such that:
\begin{itemize}
 \item $V'$ is the disjoint union of $V$ and $\{n\}$,
 \item $S'= (S\cup\{n\})\setminus\{o\}$,
 \item $E'=E\cup (E'\cap (\{o\}\times S'))$,
 \item for any $a,b\in S\setminus\{o,n\}$, $a<_S b \iff a<_{S'} b$ and $n<_{S'}a \iff o<_{S}a$.
\end{itemize}
In other words, we add to $G$ and to $S$ a new vertex $n$ that has no neighbor, we remove one vertex $o$ from $S$ 
and we can add any edges between $o$ and $S\cup\{n\}$.
Remark that the vertex removed from $S$ can be the new vertex.
The only possible change in the ordering on the distinguished vertex is that $o$ is possibly replaced by $n$ in the order (unless $o=n$ and then $<_{S'}=<_S$).

A $k$-distinguished graph has \emph{pathwidth $k$}, if it 
can be obtained from a sequence of extensions starting from $(G(S,\emptyset),S,<_S)$.
The \emph{pathwidth of a graph $G$} is the smallest integer $k$ such that $G$
can be obtained by completing a $k$-distinguished graph of pathwidth $k$.

Intuitively the set of distinguished vertices correspond to a separator of a path-decomposition of the graph.

\paragraph{Trees and forest}
In order to build trees an forests, we will use $1$-distinguished graphs (that can be seen as rooted trees or rooted forests).
The ordering on the distinguished set for $1$-distinguished graphs is trivial (since it is a singleton) and we omit it.

 Given two $1$-distinguished graphs $(G_1(V_1,E_1),\{a_1\})$ and $(G_2(V_2,E_2),\{a_2\})$ their \emph{composition} is $(G(V_1\cup V_2,E_1\cup E_2\cup\{(a_1,a_2)\}),\{a_1\})$. That is, $G$ is the disjoint union of $G_1$ and $G_2$ in which we 
add an edge between $a_1$ and $a_2$ and where the distinguished vertex is $a_1$.

We are now ready to give the following inductive definition of a \emph{$1$-tree}: 
\begin{itemize}
 \item $(G(a,\emptyset),\{a\})$ is a $1$-tree,
 \item the $1$-distinguished graph obtained by the composition of two $1$-trees is a $1$-tree.
\end{itemize}
A graph is \emph{a tree} if it can be obtained by completing a $1$-tree.

Given two $1$-distinguished graphs $(G_1(V_1,E_1),\{a_1\})$ and $(G_2(V_2,E_2),\{a_2\})$ their \emph{union} is the 
$1$-distinguished graph $(G(V,E),\{a_1\})$ where $G$ is the disjoint union of $G_1$ and $G_2$. 
Remark that the union of $1$-distinguished trees is not a commutative operation (the same can be said of the composition).

We are now ready to give the definition of \emph{$1$-forest}:
\begin{itemize}
 \item $(G(a,\emptyset),\{a\})$ is a $1$-forest,
 \item the $1$-distinguished graph obtained by the composition or the union of two $1$-forests is a $1$-forest.
\end{itemize}
A graph is a \emph{forest} if it can be obtained by completing a $1$-forest.

\section[The number of dominating sets]{Counting the number of $(\sigma,\rho)$-dominating sets}\label{secalgo}
Let $\sigma$ and $\rho$ be two recognizable sets and $k$ be a positive integer. 
In this Section, we give dynamic algorithms to compute the number of (all, minimal, maximal) $(\sigma,\rho)$-dominating sets.
The existence of these dynamic algorithms shouldn't be a surpise to anybody familiar with path-width (and tree-width) and we are not really interested by the existence of these algorithms.
However, these algorithms consist of applying (multi)linear operators 
that correspond to the operations used to build the $k$-distinguished graph. 
In the next Section, we explain how to deduce from these sets of (multi)linear operators sharp upper bounds on the number of (all, minimal, maximal) $(\sigma,\rho)$-dominating sets.

\subsection[Counting all]{Counting the $(\sigma,\rho)$-dominating sets}
The definitions being a bit obtuse one should try to follow with the two first examples of Subsection \ref{examplepw1}.

First, we generalize the notion of $(\sigma,\rho)$-dominating set to $k$-distinguished graph.
For any $k$-distinguished graph $(G(V,E),S,<_S)$ and any set $D\in V$, we say that $D$ is a $(\sigma,\rho)$-dominating
set of $(G(V,E),S,<_S)$ if 
$|N(v)\cap D|\in\sigma$ for all $v\in D\setminus S$ and $|N(v)\cap D|\in\rho$ for all $v\in V\setminus (D\cup S)$. 
That is, we do not require the elements of $S$ to be properly dominated.

For any $k$-distinguished graph $g=(G(V,E),S,<_S)$ and any extension $g'=(G'(V',E'),S',<_S')$ of $g$, if $D$ is a $(\sigma,\rho)$-dominating set  of $g'$ then
$D\cap V$ is a $(\sigma,\rho)$-dominating set of $g$. Moreover, we will show that given a $(\sigma,\rho)$-dominating set $D$ of $g'$ it is ``easy'' to check if it is or if can be extended to 
a  $(\sigma,\rho)$-dominating set of $g'$. This will allow us to compute all the  $(\sigma,\rho)$-dominating sets of $g'$ from the  $(\sigma,\rho)$-dominating sets of $g$.

Let $p_\sigma, p_\rho,q_\sigma, q_\rho\in \mathbb{N}$ be as in Lemma \ref{functionautomaton}.
Let $n=p_\sigma+p_\rho+q_\sigma+q_\rho$. 
Given a $(\sigma,\rho)$-dominating set $D$ of a $k$-distinguished graph $(G(V,E),S,<_S)$ the \emph{state corresponding to $D$} of a vertex $v$ of $V$ is
an integer $s\in[0,n-1]$ where:
$$s=\left\{
  \begin{array}{ll}
     \tau(p_\sigma,q_\sigma,|N(v)\cap D|)&\text{  if }v\in D,\\
      \tau(p_\rho,q_\rho,|N(v)\cap D|) +p_\sigma+q_\sigma& \text{ otherwise.} \\
    \end{array}
  \right.$$ 
 
The \emph{state of $S$ corresponding to $D$} is the function that maps every vertex of $v$ to its state corresponding to $D$.
The idea is that knowing the state of $S$ is enough to know what happens to a  $(\sigma,\rho)$-dominating set
when we extend the $k$-distinguished graph:
\begin{Lemma}\label{pwaddn}
Let $g=(G(V,E),S,<_S)$ and $g'=(G'(V',E'),S',<_{S'})$ be two $k$-distinguished graphs, and $o$ and $n$ be two vertices of $V'$ such that
$g'$ is obtained by extending $g$ where $n$ is the new vertex and $o$ is the vertex removed from $S\cup\{n\}$.
Let $D$ be a $(\sigma,\rho)$-dominating set of $(G(V,E),S,<_S)$.

Given the state of $S$ associated to $D$ in $(G(V,E),S,<_S)$, $E'\setminus E)$, $o$ and $n$,
one can decide if $D\cup\{n\}$ is a $(\sigma,\rho)$-dominating set of $(G'(V',E'),S',<_{S'})$ and can compute the associated state of $S'$ in $(G'(V',E'),S',<_{S'})$. 
\end{Lemma}
\begin{proof}
The new edges are only between $o$ and $S$, so any vertex from $V\setminus S$ that was properly 
$(\sigma,\rho)$-dominated by $D$ in $(G(V,E),S,<_{S})$ is still $(\sigma,\rho)$-dominated by $D\cup\{n\}$ in $(G'(V',E'),S',<_{S'})$.

If $o\not=n$ then we know the state of $o$ in $(G(V,E),S,<_{S})$,
otherwise if  $o=n$ then $o$ has no neighbor in $V\setminus S$.
So in both cases, if $o\in D$ (resp. $o\not\in D$) we know $\tau(p_\sigma,q_\sigma,|N_G(o)\cap D|)$ (resp. $\tau(p_\rho,q_\rho,|N_G(o)\cap D|)$).
But since we also know $E'\setminus E$, we can easily compute:
$$\tau(p_\sigma,q_\sigma,|N_{G'}(o)\cap D|)=\tau(p_\sigma,q_\sigma,\tau(p_\sigma,q_\sigma,|N_{G}(o)\cap D|)+|N_{G'}(v)\cap S\cap D|)$$ (or $\tau(p_\rho,q_\rho,|N_{G'}(v)\cap D|)$).
Thus we can decide if $D\cup\{n\}$ is a $(\sigma,\rho)$-dominating set of $(G'(V',E'),S',<_{S'})$.

Finally, for any vertex $v$ of $S'$:
\begin{itemize}
 \item if $(o,v)\in E'$ and $o\in D\cup \{n\}$, we can use the following equality:
 $$\tau(p,q,|N_G'(v)\cap D|)=\tau(p,q,|N_G(v)\cap D|+1)=\tau(p,q,\tau(p,q,|N_G(v)\cap D|)+1) $$
 \item otherwise, the state of $v$ in $(G'(V',E'),S',<_{S'})$ with $D$ is the same as in $(G(V,E),S,<_{S})$ with  $D\cup \{n\}$.
\end{itemize}
This concludes the proof.
\end{proof}
The essential point of this Lemma is that we need only look at $E'\setminus E$ and we need not know anything about the rest of the graph.

We can show exactly the same Lemma for $D$ instead of $D\cup\{n\}$:
\begin{Lemma}\label{pwkeepS}
Let $g=(G(V,E),S,<_{S})$ and $g'=(G'(V',E'),S',<_{S'})$ be two $k$-distinguished graphs, and $o$ and $n$ be two vertices of $V'$ such that
$g'$ is obtained by extending $g$ where $n$ is the new vertex and $o$ is the vertex removed from $S\cup\{n\}$.
Let $D$ be a $(\sigma,\rho)$-dominating set of $g$.

Given the state of $S$ associated to $D$ in $g$, $E'\setminus E$, $o$ and $n$,
one can decide if $D$  is $(\sigma,\rho)$-dominating set of $g'$ and can compute the associated state of $S'$ in $g'$. 
\end{Lemma}

Note that the distinguished set has $n^k$ possible states. 
We can order them with the lexicographic order:
a state \emph{$s_1$  of $S$ is smaller than} another state $s_2$ if
there is a vertex $v$ of $S$ such that:
\begin{itemize}
 \item for all vertices $v'<_S v$, $v'$ has the same state in $s_1$ and $s_2$,
\item the state of $v$ is smaller in $s_1$ than in $s_2$.
\end{itemize}
 
For any $k$-distinguished graph $g=(G(V,E),S,<_{S})$,
let $s_i$ be the $i$th possible state of $S$ and
let $\Psi_{\sigma,\rho}(g)\in\mathbb{N}^{n^k}$ be the vector whose $i$th coordinate is the number
of $(\sigma,\rho)$-dominating sets $D$ of $g$ such that $s_i$ corresponds to $D$.

Using Lemma \ref{pwaddn} and Lemma \ref{pwkeepS}, we can easily deduce the following one:
\begin{Lemma}\label{pwexistsmat}
 Let $g=(G(V,E),S,<_{S})$ and $g'=(G'(V',E'),S',<_{S'})$ be two $k$-distinguished graphs, and $o$ and $n$ be two vertices of $V$ such that
$g'$ is obtained by extending $g$ where $n$ is the new vertex and $o$ is the vertex removed from $S\cup\{n\}$.

Given $E'\setminus E$, $o$ and $n$,
one can compute a matrix $M$ such that $$\Psi_{\sigma,\rho}(g')=M\Psi_{\sigma,\rho}(g)\,.$$
\end{Lemma}

Once again, the crucial point of this Lemma is that $M$ only depends on $E'\setminus E$, $o$ and $n$.
Given a set $S$ there are only finitely many choices of $o$, $n$ and edges to add between $S$ and $o$.
Thus there are only finitely many possible matrices associated to an extension of a $k$-distinguished graph and one can compute all of these matrices.

Similarly, we get the following Lemma:
\begin{Lemma}
 Let $g=(G(V,E),S,<_{S})$ be a $k$-distinguished graph and $G'(V',E')$ be a graph obtained by completing $g$.

Given $E'\setminus E$, one can compute a vector $\vect{p}$ such that the number of $(\sigma,\rho)$-dominating sets of $G'$ is $\vect{p}\cdot\Psi_{\sigma,\rho}(g)$.
\end{Lemma}

We can now compute the number of $(\sigma,\rho)$-dominating sets of a graph $G$ of pathwidth at most $k$:
First, compute the sequence of extensions followed by a completion that turn $g_0=(G_0(S,\emptyset),S,>_S)$ into $G$,
then apply  to  $\Psi_{\sigma,\rho}(g_0)$ the sequence of matrices corresponding to the sequence of extensions 
and compute the dot product of the result with the vector corresponding to the completion.
Remark that in $g_0$ there is no edge yet, so a vertex has state $0$ or $p_\sigma+q_\sigma$, which implies
that we can easily find $\Psi_{\sigma,\rho}(g_0)$.

Remark that every extension increases the order of the $k$-distinguished graph by 1.
We can then deduce the following Lemma, that is crucial to bound the number of $(\sigma,\rho)$-dominating sets:
\begin{Lemma}\label{expressiontocountallpw}
Let $\sigma$ and $\rho$ be two recognizable sets of positive integers.
Then there is an integer $m$, a finite set of matrices $A_{\sigma,\rho}\subseteq \mathbb{N}^{m\times m}$, a vector $\vect{v}_{\sigma,\rho}\in  \mathbb{N}^{m}$
and a set of vectors $P_{\sigma,\rho}\subseteq\mathbb{N}^{m}$ such that:
for any integer $N$, $N$ is the number of $(\sigma,\rho)$-dominating sets of some graph of order $n>k$ if and only if there is a sequence $(M_i)_{1\le i\le n-k}\in A_{\sigma,\rho}^{n-k}$
and a vector $\vect{p}\in P_{\sigma,\rho}$ such that $N= \vect{p}\cdot (\prod_{i=1}^{n-k}M_i) \vect{v}_{\sigma,\rho}$.
\end{Lemma}

\paragraph{Trees and forests}
The idea for trees and forests is the same. 
We will use the same states and we need to be able to compute the state of the distinguished vertex after a composition or a union.
\begin{Lemma}
There exists a bilinear map $\Phi_{(\sigma,\rho)}:\mathbb{R}^n\times\mathbb{R}^n\mapsto \mathbb{R}^n$ 
such that for all $1$-distinguished graphs $g_1$ and $g_2$ and $g$ where $g$ is obtained by composition of $g_1$ and $g_2$:
$$\Psi_{\sigma,\rho}(g)=\Phi_{(\sigma,\rho)}(\Psi_{\sigma,\rho}(g_1),\Psi_{\sigma,\rho}(g_2))$$
\end{Lemma}
\begin{proof}
Let $D_1$ (resp. $D_2$) be a $(\sigma,\rho)$-dominating set of  $g_1=(G_1(V_1,E_1),\{a_1\})$ (resp. $g_2=(G_2(V_2,E_2),\{a_2\})$).
Let $g=(G(V,E),\{a_1\})$ be the composition of $g_1$ and $g_2$.
First, recall that $\tau(p,q,n+1)=\tau(p,q,\tau(p,q,n)+1)$ for all $p,n\in\mathbb{N}$ and $q\in\mathbb{N}^+$,
thus the states of $a_1$ and $a_2$ are easy to compute in $g$ with the (possibly) dominating set $D_1\cup D_2$.
Then $D_1\cup D_2$ is a $(\sigma,\rho)$-dominating set of $g$ if and only if
$a_2$ is properly $(\sigma,\rho)$-dominated by $D_1\cup D_2$ which is easy to check using its new state.
We also get the state of $a_1$.
Moreover, given a dominating set $D$, the dominating sets $D_1$ and $D_2$ always exist and are unique.

We can deduce a function $F$ such that for every pair of states of $(s_1, s_2)$:
\begin{itemize}
 \item if $D_1\cup D_2$ is not a dominating set of $g$ then $F(s_1,s_2)=-1$ with any $D_1$ and $D_2$ where $s_i$ is the state of $a_i$ corresponding to $D_i$ in $g_i$,
 \item otherwise $F(s_1,s_2)$ is the state of $a_1$ in $g$ with any dominating set $D_1\cup D_2$ where $s_i$ is the state of $a_i$ corresponding to $D_i$ in $g_i$.
\end{itemize}
Clearly the value of $F(s_1,s_2)$ only depends on the $s_i$ and not on the $D_i$.
We can now write the following formula, for any $0\le i\le n-1$:
$$\Psi_{\sigma,\rho}(g)_i=\sum_{\substack{0\le s_1,s_2\le n-1\\F(s_1,s_2)=i}} \Psi_{\sigma,\rho}(g_1)_{s_1}\Psi_{\sigma,\rho}(g_2)_{s_2}$$
Thus every coordinate of $\Psi(g)$ is bilinear in $\Psi_{\sigma,\rho}(g_1)$ and $\Psi_{\sigma,\rho}(g_2)$.
\end{proof}

Now that we have this Lemma we can describe the algorithm. 
Given a tree $G$: find the compositions that build the given tree, then apply the corresponding bilinear map, 
and finish by using the dot product with the vector corresponding to the completion.

We can deduce the following Lemma:
\begin{Lemma}\label{expressiontocountalltrees}
Let $\sigma$ and $\rho$ be two recognizable sets of positive integers.
Then there is an integer $m$,  a bilinear map $\Phi_{\sigma,\rho}:\mathbb{N}^m\times\mathbb{N}^m\mapsto \mathbb{N}^m$, and two vectors $\vect{v}_{\sigma,\rho},\vect{p}_{\sigma,\rho}\in  \mathbb{N}^{m}$
 such that:\\

If  $F_{\sigma,\rho}\subseteq (\mathbb{N}^{n}\times\mathbb{N})$ is the smallest set such that
\begin{itemize}
 \item $(\vect{v}_{\sigma,\rho} , 1)\in F_{\sigma,\rho}$,
 \item if $(v,i),(u,j)\in F_{\sigma,\rho}$ then $(\Phi_{\sigma,\rho}(u,v),i+j)\in F_{\sigma,\rho}$.
\end{itemize}
then there is a tree of order $n$ that admits $N$ $(\sigma,\rho)$-dominating sets if and only if
there exists $u$ such that $(u,n)\in F_{\sigma,\rho}$ and $N= \vect{p}_{\sigma,\rho}\cdot u$.
\end{Lemma}

Similarly, one can compute a bilinear map for unions:
\begin{Lemma}
There exists a bilinear map $\Delta_{(\sigma,\rho)}:\mathbb{R}^n\times\mathbb{R}^n\mapsto \mathbb{R}^n$ 
such that for all $1$-distinguished graphs $g_1$ and $g_2$ and $g$ where $g$ is obtained by union of $g_1$ and $g_2$:
$$\Psi_{\sigma,\rho}(g)=\Delta_{(\sigma,\rho)}(\Psi_{\sigma,\rho}(g_1),\Psi_{\sigma,\rho}(g_2))$$
\end{Lemma}
And we can deduce the following Lemma:
\begin{Lemma}\label{expressiontocountallforests}
Let $\sigma$ and $\rho$ be two recognizable sets of positive integers.
Then there is an integer $m$, two vectors $\vect{v}_{\sigma,\rho},\vect{p}_{\sigma,\rho}\in  \mathbb{N}^{m}$
and two bilinear maps $\Phi_{\sigma,\rho},\Delta_{\sigma,\rho}\in (\mathbb{N}^m)^{\mathbb{N}^m\times\mathbb{N}^m}$, such that:\\

If  $F_{\sigma,\rho}\subseteq (\mathbb{N}^{n}\times\mathbb{N})$ is the smallest set such that
\begin{itemize}
 \item $(\vect{v}_{\sigma,\rho} , 1)\in F_{\sigma,\rho}$,
 \item if $(v,i),(u,j)\in F_{\sigma,\rho}$ then $(\Phi_{\sigma,\rho}(u,v),i+j)\in F_{\sigma,\rho}$ and $(\Delta_{\sigma,\rho}(u,v),i+j)\in F_{\sigma,\rho}$.
\end{itemize}
then there is a forest of order $n$ that admits $N$ $(\sigma,\rho)$-dominating sets if and only if
there exists $u$ such that $(u,n)\in F_{\sigma,\rho}$ and $N= \vect{p}_{\sigma,\rho}\cdot u$.
\end{Lemma}

\subsection[Counting min]{Counting the 1-minimal $(\sigma,\rho)$-dominating sets}
The idea for counting the number of 1-minimal $(\sigma,\rho)$-dominating sets is similar, 
but we will need more states in order to know if a vertex should really be on the $(\sigma,\rho)$-dominating set.

First we need to introduce the notion of \emph{certificate}.
Let $G$ be a graph and $D$ be a $(\sigma,\rho)$-dominating set of $G$ the set $C$
of certificates is the set of vertices that would not be dominated when losing one neighbor in $D$, that is:
$C=\{v\in D: |N(v)\cap D|-1\not\in \sigma\}\cup\{v\not\in D: |N(v)\cap D|-1\not\in \rho\}$.
We say that a vertex $v\in D$ \emph{has a certificate} if one of its neighbor is a certificate. 
It is \emph{is self-certified} if $|N(v)\cap D|\not\in\rho$.

The following Lemma should be clear from the definitions:
\begin{Lemma}
Let $G$ be a graph and $D$ be a $(\sigma,\rho)$-dominating set of $G$.
Then $D$ is a 1-minimal $(\sigma,\rho)$-dominating set of $G$ if and only if for any vertex $v\in D$, $v$ is self-certified or has a certificate.
\end{Lemma}

We want to generalize the notion of certificate to $k$-distinguished graphs, with the previous Lemma in mind.
A set $C$ is a \emph{set of certificates of a $(\sigma,\rho)$-dominating set} of $(G(V,E), S)$ if all the vertices of $C\setminus S$ are certificates.
The idea is that the neighborhood of any element of $S$ may change, so we predict from the start if it is or not a certificate and when the vertex leaves $S$ we check if it respects the prediction.

Let $g=(G(V,E), S)$ be a $k$ distinguished graph, $C$ be a set of certificates of $D$ a $(\sigma,\rho)$-dominating set of $g$.
Then we say that $D$ is a $(1,C)$-minimal $(\sigma,\rho)$-dominating set of $g$ if all the vertices of $D\setminus S$ have a neighbor in $C$ or are self-certified.
For the sake of brevity we will say that $(D,C)$ is \emph{a good pair} of $g$ if $D$ is a $(1,C)$-dominating set of $g$ and
$C$ is a set of certificates of $D$.

Given a $k$ distinguished graph $g$ and a good pair $(D,C)$ of $g$
it is easy to decide whether $D$ is a 1-minimal $(\sigma,\rho)$-dominating set of a given completion of $g$ and if $C$ is the corresponding set of certificates.
In order to make that more precise we can finally introduce the states.

Let $\sigma,\rho\subseteq\mathbb{N}$ and $p$ and $q$ be as in Lemma \ref{recog}.
For any $k$ distinguished graph $g=(G(V,E), S)$ and good pair $(D,C)$ of $g$,
the state of a vertex $v$  that \emph{corresponds} to $(D,C)$ is the triplet $$\left(1_D(v)\cdot(1+1_{N(v)\cap C\not=\emptyset}), 1_C(v),  \tau(p,q,|N(v)\cap D|)\right)\,.$$
Let us give meaning to the $3$ possible values of $1_D(v)\cdot(1+1_{N(v)\cap C\not=\emptyset})$:
\begin{enumerate}
\setcounter{enumi}{-1}
 \item $v$ is not in $D$,
 \item $v$ is in $D$ but has no certificate,
 \item $v$ is in $D$ and has a certificate.
\end{enumerate}

The state of $S$ is the function that maps every vertex from $S$ to its state.
As in the previous section, we can define a strict total ordering on the states of $S$ based on $>_S$ and on the lexicographical order to define a vector:
$\Psi_{\sigma,\rho,\min}(g)$ whose $i$th coordinate counts the number of good pairs of $g$ that correspond to the $i$th state of $S$.
We do not give the details on the chosen ordering since it can be any strict total ordering.

We can now give the following Lemma:
\begin{Lemma}
Let $g=(G(V,E),S,<_S))$ and $g'=(G'(V',E'),S',<_{S'}))$ be two $k$-distinguished graphs, and $o$ and $n$ be two vertices of $V'$ such that
$g'$ is obtained by extending $g$ where $n$ is the new vertex and $o$ is the vertex removed from $S\cup\{n\}$.

Given $E'\setminus E$ one can compute a matrix $M$ such that $$\Psi_{\sigma,\rho,\min}(g')=M\Psi_{\sigma,\rho,\min}(g).$$
\end{Lemma}
The idea behind this Lemma is really similar to Lemma \ref{pwexistsmat} so we sketch the proof.
\begin{proof}
Given a state of $S$ with a good pair $(D,C)$, only few things change when extending:
\begin{itemize}
 \item if $o$ is in $D$ then the last coordinate of the state of each of its new neighbors in $S'$ is incremented,
 \item if $o$ is in $C$ then the first coordinate of the state of each of its new neighbors become $2$ if it was $1$,
 \item the last coordinate of the state of $o$ is increased by $1$ for each new neighbor in $D$,
 \item the first coordinate of the state of $o$ is set to 2 if it was $2$ and if $o$ has a new neighbor in $C$.
\end{itemize}
Thus the new states of the elements of $S'$ and of $o$ only depends on the former states of the element of $S$ and on the set of new edges.
To check if the new pair is still a good pair, one can check that $o$  has a certificate or is self-certified if it belongs to $D$, is properly dominated and a certificate if and only if its state says so.
This can be verified by checking the following conditions:
\begin{itemize}
 \item if the first coordinate of $o$ is $1$ then $\tau(p,q,|N(o)\cup D|)\not\in \rho$,
 \item if the first coordinate of $o$ is $0$ (resp. $2$ or $1$) then $\tau(p,q,|N(o)\cup D|)\in \rho$ (resp. $\tau(p,q,|N(o)\cup D|)\in \sigma$ ),
 \item if the first coordinate of $o$ is $0$ (resp. $2$ or $1$) $1_C(o)=1$ if and only if $\tau(p,q,|N(o)\cup D|)-1\not\in \rho$ (resp. $\tau(p,q,|N(o)\cup D|)-1\not\in \sigma$ ).
\end{itemize}
So the state of $S'$ after an extension only depends on the state of $S$ before the extension a and on the set of new edges.
It is then easy to deduce the matrix $M$.
\end{proof}

Similarly we have the following Lemma:
\begin{Lemma}
 Let $g=(G(V,E),S,<_{S})$ be a $k$-distinguished graph and $G'(V',E')$ be a graph obtained by completing $g$.

Given $E'\setminus E$, one can compute a vector $\vect{p}$ such that the number of 1-minimal $(\sigma,\rho)$-dominating set of $G'$ is $\vect{p}\cdot\Psi_{\sigma,\rho,\min}(g)$.
\end{Lemma}

In order to compute the number of $1$-minimal $(\sigma,\rho)$-dominating sets of a graph $G$ of pathwidth at most $k$, 
one can find a sequence of extensions to apply to $g_0=(G_0(S,\emptyset),S,>_S)$ followed by a completion that gives $G$.
Then the number of $1$-minimal $(\sigma,\rho)$-dominating sets can be obtained by multiplying $\Psi_{\sigma,\rho,\min}(g_0)$ by the matrices corresponding
to the extensions applied and finally compute the dot product with the vector corresponding to the completion.

We deduce the following Lemma: 
\begin{Lemma}\label{expressiontocountminpw}
Let $\sigma$ and $\rho$ be two recognizable sets of positive integers.
Then there is an integer $m$, a finite set of matrices $A_{\sigma,\rho,\min}\subseteq \mathbb{N}^{m\times m}$, a vector $\vect{v}_{\sigma,\rho,\min}\in  \mathbb{N}^{m}$
and a set of vectors $P_{\sigma,\rho,\min}\subseteq\mathbb{N}^{m}$ such that:
for any integer $N$, $N$ is the number of 1-minimal $(\sigma,\rho)$-dominating sets of some graph of order $n>k$ if and only if there is a sequence $(M_i)_{1\le i\le n-k}\in A_{\sigma,\rho,\min}^{n-k}$
and a vector $\vect{p}\in P_{\sigma,\rho,\min}$ such that $N= \vect{p}\cdot (\prod_{i=1}^{n-k}M_i) \vect{v}_{\sigma,\rho,\min}$.
\end{Lemma}

\paragraph{Trees and forests}
The idea for trees and forests is once again the same.
We could show the following Lemmas:
\begin{Lemma}
There exists a bilinear map $\Phi_{(\sigma,\rho,\min)}:\mathbb{R}^n\times\mathbb{R}^n\mapsto \mathbb{R}^n$ 
such that for all $1$-distinguished graphs $g_1$ and $g_2$ and $g$ where $g$ is obtained by composition of $g_1$ and $g_2$:
$$\Psi_{\sigma,\rho,\min}(g)=\Phi_{(\sigma,\rho,\min)}(\Psi_{\sigma,\rho,\min}(g_1),\Psi_{\sigma,\rho,\min}(g_2))$$
\end{Lemma}
\begin{Lemma}
There exists a bilinear map $\Delta_{(\sigma,\rho,\min)}:\mathbb{R}^n\times\mathbb{R}^n\mapsto \mathbb{R}^n$ 
such that for all $1$-distinguished graphs $g_1$ and $g_2$ and $g$ where $g$ is obtained by union of $g_1$ and $g_2$:
$$\Psi_{\sigma,\rho,\min}(g)=\Delta_{(\sigma,\rho,\min)}(\Psi_{\sigma,\rho,\min}(g_1),\Psi_{\sigma,\rho,\min}(g_2))$$
\end{Lemma}
Then the number of $1$-minimal $(\sigma,\rho)$-dominating sets of a graph $G$ 
can be computed by applying the operation corresponding to the decomposition of $G$.

We deduce the following Lemmas:
\begin{Lemma}\label{expressiontocountmintree}
Let $\sigma$ and $\rho$ be two recognizable sets of positive integers.
Then there is an integer $m$, two vectors $\vect{v}_{\sigma,\rho,\min},\vect{p}_{\sigma,\rho,\min}\in  \mathbb{N}^{m}$
and a bilinear map $\Phi_{\sigma,\rho,\min}:\mathbb{N}^m\times\mathbb{N}^m\mapsto \mathbb{N}^m$ such that:\\

If  $F_{\sigma,\rho,\min}\subseteq (\mathbb{N}^{n}\times\mathbb{N})$ is the smallest set such that
\begin{itemize}
 \item $(\vect{v}_{\sigma,\rho,\min} , 1)\in F_{\sigma,\rho,\min}$,
 \item if $(v,i),(u,j)\in F_{\sigma,\rho,\min}$ then $(\Phi_{\sigma,\rho,\min}(u,v),i+j)\in F_{\sigma,\rho,\min}$.
\end{itemize}
then there is a tree of order $n$ that admits $N$ $(\sigma,\rho)$-dominating sets if and only if
there exists $u$ such that $(u,n)\in F_{\sigma,\rho,\min}$ and $N= \vect{p}_{\sigma,\rho,\min}\cdot u$.
\end{Lemma}
\begin{Lemma}\label{expressiontocountminforest}
Let $\sigma$ and $\rho$ be two recognizable sets of positive integers.
Then there is an integer $m$, two vectors $\vect{v}_{\sigma,\rho,\min},\vect{p}_{\sigma,\rho,\min}\in  \mathbb{N}^{m}$
and two bilinear maps $\Phi_{\sigma,\rho,\min}:\mathbb{N}^m\times\mathbb{N}^m\mapsto \mathbb{N}^m$ and :
$\Delta_{(\sigma,\rho,\min)}:\mathbb{N}^m\times\mathbb{N}^m\mapsto \mathbb{N}^m$, such that:\\

If  $F_{\sigma,\rho,\min}\subseteq (\mathbb{N}^{n}\times\mathbb{N})$ is the smallest set such that
\begin{itemize}
 \item $(\vect{v}_{\sigma,\rho,\min} , 1)\in F_{\sigma,\rho,\min}$,
 \item for any $(v,i),(u,j)\in F_{\sigma,\rho,\min}$ then $(\Phi_{\sigma,\rho,\min}(u,v),i+j)\in F_{\sigma,\rho,\min}$ and $(\Delta_{\sigma,\rho,\min}(u,v),i+j)\in F_{\sigma,\rho,\min}$.
\end{itemize}
then there is a forest of order $n$ that admits $N$ $(\sigma,\rho)$-dominating sets if and only if
there exists $u$ such that $(u,n)\in F_{\sigma,\rho,\min}$ and $N= \vect{p}_{\sigma,\rho,\min}\cdot u$.
\end{Lemma}

\subsection[Counting max]{Counting the 1-maximal $(\sigma,\rho)$-dominating sets}
For any graph $G(V,E)$ and $(\sigma,\rho)$-dominating sets $D$, 
$D$ is 1-maximal if $V\setminus D$ is 1-minimal.
So we can define a notion of certificate for the vertices of $V\setminus D$ and use the sames ideas as in the previous Section.
A vertex is a certificate if is not properly dominated when we give him one more vertex in $D$, that is, the set of certificates is:
$\{v\in D: |N(v)\cap D|+1\not\in \sigma\}\cup\{v\not\in D: |N(v)\cap D|+1\not\in \rho\}$. 
Then a $(\sigma,\rho)$-dominating set $D$ of $G(V,E)$ is 1-maximal if and only if any vertex from $V\setminus D$ has a certificate among its neighbors or is self-certified ($|N(v)\cap D|\not\in \sigma$).

We can then generalize this notion of certificate to $k$-distinguished graphs and show the following Lemmas:
\begin{Lemma}\label{expressiontocountmaxpw}
Let $\sigma$ and $\rho$ be two recognizable sets of positive integers.
Then there is an integer $m$, a finite set of matrices $A_{\sigma,\rho,\max}\subseteq \mathbb{N}^{m\times m}$, a vector $\vect{v}_{\sigma,\rho,\max}\in  \mathbb{N}^{m}$
and a set of vectors $P_{\sigma,\rho,\max}\subseteq\mathbb{N}^{m}$ such that:
for any integer $N$, $N$ is the number of 1-minimal $(\sigma,\rho)$-dominating sets of some graph of order $n>k$ if and only if there is a sequence $(M_i)_{1\le i\le n-k}\in A_{\sigma,\rho,\max}^{n-k}$
and a vector $\vect{p}\in P_{\sigma,\rho,\max}$ such that $N= \vect{p}\cdot (\prod_{i=1}^{n-k}M_i) \vect{v}_{\sigma,\rho,\max}$.
\end{Lemma}

\begin{Lemma}\label{expressiontocountmaxtree}
Let $\sigma$ and $\rho$ be two recognizable sets of positive integers.
Then there is an integer $m$, two vectors $\vect{v}_{\sigma,\rho,\max},\vect{p}_{\sigma,\rho,\max}\in  \mathbb{N}^{m}$
and a bilinear map $\Phi_{\sigma,\rho,\max}:\mathbb{N}^m\times\mathbb{N}^m\mapsto \mathbb{N}^m$, such that:\\

If  $F_{\sigma,\rho,\max}\subseteq (\mathbb{N}^{n}\times\mathbb{N})$ is the smallest set such that
\begin{itemize}
 \item $(\vect{v}_{\sigma,\rho,\max} , 1)\in F_{\sigma,\rho,\max}$,
 \item if $(v,i),(u,j)\in F_{\sigma,\rho,\max}$ then $(\Phi_{\sigma,\rho,\max}(u,v),i+j)\in F_{\sigma,\rho,\max}$.
\end{itemize}
then there is a tree of order $n$ that admits $N$ $(\sigma,\rho)$-dominating sets if and only if
there exists $u$ such that $(u,n)\in F_{\sigma,\rho,\max}$ and $N= \vect{p}_{\sigma,\rho,\max}\cdot u$.
\end{Lemma}

\begin{Lemma}\label{expressiontocountmaxforest}
Let $\sigma$ and $\rho$ be two recognizable sets of positive integers.
Then there is an integer $m$, two vectors $\vect{v}_{\sigma,\rho,\max},\vect{p}_{\sigma,\rho,\max}\in  \mathbb{N}^{m}$
and two bilinear maps $\Phi_{\sigma,\rho,\max}:\mathbb{N}^m\times\mathbb{N}^m\mapsto \mathbb{N}^m$ and  $\Delta_{\sigma,\rho,\max}:\mathbb{N}^m\times\mathbb{N}^m\mapsto \mathbb{N}^m$, such that:\\

If  $F_{\sigma,\rho,\max}\subseteq (\mathbb{N}^{n}\times\mathbb{N})$ is the smallest set such that
\begin{itemize}
 \item $(\vect{v}_{\sigma,\rho,\max} , 1)\in F_{\sigma,\rho,\max}$,
 \item for any $(v,i),(u,j)\in F_{\sigma,\rho,\max}$ then $(\Phi_{\sigma,\rho,\max}(u,v),i+j)\in F_{\sigma,\rho,\max}$ and  $(\Delta_{\sigma,\rho,\max}(u,v),i+j)\in F_{\sigma,\rho,\max}$.
\end{itemize}
then there is a forest of order $n$ that admits $N$ $(\sigma,\rho)$-dominating sets if and only if
there exists $u$ such that $(u,n)\in F_{\sigma,\rho,\max}$ and $N= \vect{p}_{\sigma,\rho,\max}\cdot u$.
\end{Lemma}

\section{The bounds}\label{secbounds}
\subsection{Graph of bounded pathwidth}
Let $\sigma$ and $\rho$ be two recognizable sets of positive integers and 
$\#_{\sigma,\rho}(n)$ be the maximal number of $(\sigma,\rho)$-dominating sets (resp. 1-minimal $(\sigma,\rho)$-dominating sets,  1-maximal $(\sigma,\rho)$-dominating sets) in graph of pathwidth at most $k$ of order $n$.

Then we can use Lemma \ref{expressiontocountallpw}(resp. Lemma \ref{expressiontocountminpw}, Lemma \ref{expressiontocountmaxpw})
to find an integer $m$, a set of matrices $A\subseteq \mathbb{N}^{m\times m}$,  a vector $\vect{v}_{\sigma,\rho}\in\mathbb{N}^{m}$ and a set of
vectors $P\subseteq \mathbb{N}^{m}$ such that for any integer $N$ that corresponds to 
the number of (all, 1-minimal, 1-maximal) $(\sigma,\rho)$-dominating sets of a graph of patwidth $k$ and order $n>k$ there is a sequence $(M_i)_{1\le i\le n-k}\in A^{n-k}$
and a vector $\vect{p}\in P$ such that $N= \vect{p}\cdot (\prod_{i=1}^{n-k}M_i) \vect{v}_{\sigma,\rho}$.
Thus in particular:
$$\#_{\sigma,\rho}(n) =\max \left\{ \vect{p}\cdot \left(\prod_{i=1}^{n-k}M_i\right) \vect{v}_{\sigma,\rho}:\vect{p}\in P , (M_i)_{1\le i\le n-k}\in A^{n-k}\right\}.$$
The goal of this subsection is to explain how we bound this quantity.

First, we need to simplify the set of matrices.
We say that the $i$-th coordinate is \emph{accessible} if there is a sequence $(M_i)_{1\le i\le n-k}\in A^{n-k}$ such that the $i$-th coordinate
of $(\prod_{i=1}^{n-k}M_i) \vect{v}_{\sigma,\rho}$ is non-zero. 
Similarly the $i$th coordinate is \emph{co-accessible} if there is a sequence $(M_i)_{1\le i\le n-k}\in A^{n-k}$ and $\vect{p}\in P$ such that 
$\vect{p}\cdot(\prod_{i=1}^{n-k}M_i) \vect{v}_{\sigma,\rho}$ is non-zero. 
Note that the sets of accessible and co-accessible coordinates can easily be computed by a recursive algorithm.
The motivation of these definitions is that we can ignore coordinates that are not accessible or not co-accessible since they do not influence the result of the product.

Let $\widetilde{m}$ be the number of accessible and co-accessible coordinates.
Let $C\in\mathbb{N}^{\widetilde{m}\times m}$ be the matrix that maps the $i$-th coordinate to the $i$-th accessible and co-accessible coordinate and let:
$\widetilde{A}=\{CMC^T: M\in A\}$, $\widetilde{\vect{v}_{\sigma,\rho}}=C\vect{v}_{\sigma,\rho}$ and $\widetilde{P}=\{C\vect{p}: \vect{p}\in P\}$.
Then by definition:
$$\#_{\sigma,\rho}(n) =\max \left\{ \vect{p}\cdot \left(\prod_{i=1}^{n-k}M_i\right) \widetilde{\vect{v}_{\sigma,\rho}}:\vect{p}\in \widetilde{P} , (M_i)_{1\le i\le n-k}\in \widetilde{A}^{n-k}\right\}.$$

We can finally explain how to compute this quantity.
For any set of points $X$, we denote by $\operatorname{conv}(X)$ the convex hull of $X$.
We can now give the following Theorem:
\begin{Theorem}\label{convexhull}
Let $\alpha$ be a positive real, $A\subseteq\mathbb{R^+}^{m\times m}$ be a set of matrices and $\vect{v}_{\sigma,\rho}\in\mathbb{R^+}^m$.
Let $X\subseteq\mathbb{R^+}^m$ be a bounded set of vectors such that:
\begin{itemize}
 \item $\vect{v}_{\sigma,\rho}\in\operatorname{conv}(X)$,
 \item $\forall \vect{x}\in X$, for all $M\in A$, $\frac{1}{\alpha}M \vect{x}\in\operatorname{conv}(X)$.
\end{itemize}
Then for any $\vect{p}\in\mathbb{R^+}^m$ there exists a constant $C$ such that for all integers $n$
$$\max \left\{ \left|\vect{p}\cdot\left(\prod_{i=1}^{n-k}M_i\right) \vect{v}_{\sigma,\rho}\right|:(M_i)_{1\le i\le n-k}\in A^{n-k}\right\}< C\alpha^n$$
\end{Theorem}
\begin{proof}
 First remark that $\forall \vect{x}\in\operatorname{conv}(X)$, for all $M\in A$, $\frac{1}{\alpha}M \vect{x}\in\operatorname{conv}(X)$.
 Thus by induction on $n$, for any sequence $(M_i)_{1\le i\le n-k}\in A^{n-k}$, 
 $$\left(\prod_{i=1}^{n-k}\frac{1}{\alpha}M_i\right) \vect{v}_{\sigma,\rho}\in\operatorname{conv}(X).$$
 This implies that:
 \begin{align*}
  \left|\vect{p}\cdot\left(\prod_{i=1}^{n-k}\frac{1}{\alpha}M_i\right) \vect{v}_{\sigma,\rho}\right|&< \max_{\vect{x}\in \operatorname{conv}(X)} |\vect{p}\cdot \vect{x}|\\
  \left|\vect{p}\cdot\left(\prod_{i=1}^{n-k}\frac{1}{\alpha}M_i\right) \vect{v}_{\sigma,\rho}\right|&< \max_{\vect{x}\in X} |\vect{p}\cdot \vect{x}|\\
 \left|\vect{p}\cdot\left(\prod_{i=1}^{n-k}M_i\right) \vect{v}_{\sigma,\rho}\right|&<\alpha^{n} \max_{\vect{x}\in X} \frac{|\vect{p}\cdot \vect{x}|}{\alpha^k}
 \end{align*}
 This concludes the proof.
\end{proof}

For any $X\in\mathbb{R^+}^m$, let $\operatorname{conv}_\le(X)= \{\vect{x}\in\mathbb{R^+}^m: \vect{x'}\in \operatorname{conv}(X), \vect{x}\le\vect{x'}\}$.
Remark that $\operatorname{conv}_\le(X)$ is also a convex set so we get the following Corollary:
\begin{Corollary}\label{convexinfhull}
Let $\alpha$ be a positive real, $A\subseteq\mathbb{R^+}^{m\times m}$ be a set of matrices and $\vect{v}_{\sigma,\rho}\in\mathbb{R^+}^m$.
Let $X\subseteq\mathbb{R^+}^m$ be a bounded set of vectors such that:
\begin{itemize}
 \item $\vect{v}_{\sigma,\rho}\in\operatorname{conv_<}(X)$,
 \item $\forall \vect{x}\in X$, for all $M\in A$, $\frac{1}{\alpha}M \vect{x}\in\operatorname{conv_<}(X)$.
\end{itemize}
Then for any $\vect{p}\in\mathbb{R^+}^m$ there exists a constant $C$ such that for all integers $n$
$$\max \left\{ \left|\vect{p}\cdot\left(\prod_{i=1}^{n-k}M_i\right) \vect{v}_{\sigma,\rho}\right|:(M_i)_{1\le i\le n-k}\in A^{n-k}\right\}< C\alpha^n$$
\end{Corollary}
Remark that using linear programming it is not significantly more expensive to check whether a point belongs to $\operatorname{conv}_\le(X)$ instead of $\operatorname{conv}(X)$.
However, it seems that in many cases it would give us much smaller set $X$.

Also remark, that the condition $\vect{v}_{\sigma,\rho}\in\operatorname{conv_<}(X)$ could be replaced by the condition: for all  $1\le i\le m$, there exists $v\in X$ such that the $i$th coordinate of $v$ is positive 
(or equivalently: $\operatorname{conv}_\le(X)$ spans $\mathbb{R}^m$). 
In deed, we can scale $X$ without changing the result,
and as long as the polytope is of dimension $m$ we can scale it enough to contain $\vect{v}_{\sigma,\rho}\in\operatorname{conv_<}(X)$.
We do not explicitely use that remark, but it helps to find set $X$ by hand. 
This remark, will not hold for trees and forests  (Theorem \ref{convexinfhullTree}), because in these cases we cannot scale the set $X$.

In fact, in the case of pathwidth, we can replace the constant $C$ by $1$ using the following Lemma:
\begin{Lemma}\label{feketesApplication}
If there exists $\alpha$ and $C$ such that for all $n$, $\#_{\sigma,\rho}(n)<C\alpha^n$ then for all $n$:
$\#_{\sigma,\rho}(n)<\alpha^n$.
\end{Lemma}
\begin{proof}
Remark that the number of (1-minimal, 1-maximal) $(\sigma,\rho)$-dominating sets of the disjoint union of two graphs is the product of the numbers of (1-minimal, 1-maximal) $(\sigma,\rho)$-dominating sets of the two graphs.
Now suppose by contradiction that there is an integer $N$ and a positive real $\varepsilon$ such that:
$\#_{\sigma,\rho}(N)\ge\alpha^N+\varepsilon$, then by the previous remark that would mean that for any integer $n$
$\#_{\sigma,\rho}(nN)\ge(\alpha^N+\varepsilon)^n$.
Then $\#_{\sigma,\rho}(n)$ can't be bounded by $C\alpha^n$ which is a contradiction.
\end{proof}

Nothing says that there always is such a set $X$ that we can compute.
But, in practice, we can often find a finite set $X$ using the following trivial algorithm:

\begin{algorithm}[H]\label{algopw}
 \KwData{$\widetilde{A}, \widetilde{\vect{v}_{\sigma,\rho}}$, $\alpha$}
 \KwResult{A set $X$ that respects the conditions of Corollary \ref{convexinfhull}}
 $X:=\{\widetilde{\vect{v}_{\sigma,\rho}}\}$\;
  \While{$\exists M\in \widetilde{A}, \exists \vect{x}\in X$  such that $\frac{1}{\alpha}Ms\not\in \operatorname{conv_<}(X)$}{
    $X':=\{\frac{1}{\alpha}M\vect{x}: M\in \widetilde{A}, \vect{x}\in X\}$\;
    $X := \operatorname{Hull}_\le(X\cup X')$\;
 }
\caption{Computation of the set $X$}
\end{algorithm}

$\operatorname{Hull}_\le$ is a function that compute the smallest subset of points such that $\operatorname{conv}_\le(X)=\operatorname{conv}_\le(\operatorname{Hull}_\le(X))$.
One naive way to do that is to simply check for each point $\vect{x}\in X$ whether or not $x\in\operatorname{conv}_\le(X\setminus\{x\})$, which can easily be done with linear programming.
The algorithm does not necessarily terminate, but if it does it returns a set $X$ corresponding to the one of Corollary \ref{convexinfhull}.

Although the algorithm is trivial, it is important to have an efficient implementation since 
computing the convex hull of a set of points in high dimension is expensive (in particular when the coordinates are not rationals, but algebraic numbers).
The joint C++ code  described in Annexe generates $\widetilde{A}$, $\widetilde{P}$ and $\widetilde{\vect{v}_{\sigma,\rho}}$ for a given choice  of $\sigma$ and $\rho$ and
then apply Algorithm \ref{algopw} with a given choice of $\alpha$.

In fact, because of the combinatorial explosion, it probably requires a lot of work in order to make this technique work for a given choice of $(\sigma,\rho)$ in pathwidth 3 or more.
There are many tricks that we do not discuss in details that we could take advantage of. 
If a matrix from $\widetilde{A}$ is smaller in every coordinate than another one, then we can ignore this matrix (this was used in our implementation).
Another idea it to take another order $<$ in the definition of $\operatorname{conv}_\le$ (instead of the coordinate-wise comparison) on the vectors and everything still works as long as left multiplication by the matrices of $A$ is monotone with respect to $<$.
This last idea was used in the specific case of minimal dominating set (and is called \emph{majorization}) in Section 6.1 of \cite{r-mnmds-19,arxivRote}.

Finally, while it's great to find an upper bound, it is better to know that it is sharp.
A construction that reaches the bound is the best way to know that.
If running with $\alpha$ slightly smaller than the conjectured bound the algorithm will not terminate,
however by inspecting the sequences of product of matrix that give the extremal points one can deduce such constructions 
(it can be partly automated, and it seems much more efficient than a brute force approach). 
In fact, this is how we found all the sharp examples on more than 3 vertices.

\subsubsection{Results and examples in pathwidth 1}\label{examplepw1}
Note that in the case of pathwidth $1$ we have 4 possible ways of extending a $k$-distinguished graph:
we can keep either the new or the old node in $S$, and they can share an edge or not.
Thus we have 4 possible matrices that we will give in the following order, where $v$ is the vertex that was already distinguished and
$n$ is the new vertex:
\begin{enumerate}
 \item we keep $v$ and there is an edge between $v$ and $n$,
 \item we keep $v$ and there is no edge between $v$ and $n$,
 \item we keep $n$ and there is an edge between $v$ and $n$,
 \item we keep $n$ and there is no edge between $v$ and $n$.
\end{enumerate}

We will detail the computations in the case of induced matching and of perfect total dominating sets to serve as an example.

\paragraph[Induced matching]{Induced matching:}$\sigma=\{1\}$, $\rho=\mathbb{N}$

First we can compute $p_\sigma=2$, $q_\sigma=1$, $p_\rho=0$ and $q_\rho=1$. 
There are 4 possible states for the distinguished vertex $s$ and they correspond to:
\begin{enumerate}
\setcounter{enumi}{-1}
 \item $s$ is in $D$ and has no neighbor in $D$,
 \item $s$ is in $D$ and has one neighbor in $D$,
 \item $s$ is in $D$ and has at least two neighbors in $D$,
 \item $s$ is not in $D$.
\end{enumerate}

Since $n$ is the new vertex, it is either in state $0$ or $3$ before we possibly add the edge.
For the same reason, $\vect{v}_{\sigma,\rho}=(1,0,0,1)^T$.
Moreover, to be properly dominated a vertex has to be in state $1$ or $3$, thus: $P=\{(0,1,0,1)^T \}.$

Let us detail the computation of the matrices.
We start with the first matrix, that is, we keep $v$ and there is an edge between $v$ and $n$.
If $n$ is in state $3$ then it stays in state $3$ and is properly dominated and the state of $v$ does not change.
Otherwise if $n$ is in state $0$:
\begin{description}
 \item [$v$ is in state $0$:] then $n$ and $v$ are now in state $1$ so we can forget $n$,
 \item [$v$ is in state $1$ or $2$:] then $n$ is now in state $1$ so we can forget it, and $v$ is now in state $2$
 \item [$v$ is in state $3$:] then $n$ is still in state $0$, so we cannot forget it and there is no corresponding $(\sigma,\rho)$-dominating set.
\end{description}
It gives the following matrix:
$$M_1=\begin{pmatrix}
1&0&0&0\\
1&1&0&0\\
0&1&2&0\\
0&0&0&1\\
\end{pmatrix}.$$

Now, if there no edge between $n$ and $v$ and we forget $v$, $n$ has to be in state $3$ and the state of $v$ does not changes which gives:
$$M_2=\begin{pmatrix}
1&0&0&0\\
0&1&0&0\\
0&0&1&0\\
0&0&0&1\\
\end{pmatrix}.$$

Now, if there is an edge between $n$ and $v$ and we forget $n$:
\begin{description}
 \item [$v$ is in state $0$:]~
\begin{description}
  \item [$n$ is in state $0$:] $v$ and $n$ are now in state $1$ and we can forget $v$,
  \item [$n$ is in state $3$:] $v$ is still in state $0$ and we can't forget it.
 \end{description}
 \item [$v$ is in state $1$:]~
\begin{description}
  \item [$n$ is in state $0$:] $v$ is now in state $2$ and we can't forget it,
  \item [$n$ is in state $3$:] $v$ stays in state $1$ and we can forget it, and $n$ stays in state $3$.
 \end{description} 
 \item [$v$ is in state $2$:] $v$ stays in state $2$ so we can't forget it.
 \item [$v$ is in state $3$:] $v$ stays in state $3$ and $n$ can take state $0$ or $3$.
\end{description}
It gives the following matrix:
$$M_3=\begin{pmatrix}
0&0&0&1\\
1&0&0&0\\
0&0&0&0\\
0&1&0&1\\
\end{pmatrix}.$$

Finally, if there no edge between $n$ and $v$ and we forget $n$, then $v$ can be forgotten iff it is already in state $1$ or $3$ and $n$ can take state $0$ or $3$.
It give the following matrix:
$$M_4=\begin{pmatrix}
0&1&0&1\\
0&0&0&0\\
0&0&0&0\\
0&1&0&1\\
\end{pmatrix}.$$

Let $\vect{e}_2=(0,0,1,0)^T$.
For any $M_i$, $M_i\vect{e}_2\in \{\vect{0},\vect{e}_2,2\vect{e}_2\}$ and moreover for all $p\in P$, $p\cdot \vect{e}_2=0$.
It implies that the third coordinate is not co-accessible.
We could have noticed earlier that if a vertex is in state $2$,
it has too many neighbors in the $(\sigma,\rho)$-dominating set and so the corresponding $(\sigma,\rho)$-dominating set cannot be extended.

It is easy to check that all the other coordinates are accessible and co-accessible, and we get:
$$\resizebox{0.7\hsize}{!}{ $\widetilde{A}=
\left\{\begin{pmatrix}
1&0&0\\
1&1&0\\
0&0&1\\
\end{pmatrix},
\begin{pmatrix}
1&0&0\\
0&1&0\\
0&0&1\\
\end{pmatrix},
\begin{pmatrix}
0&0&1\\
1&0&0\\
0&1&1\\
\end{pmatrix},
\begin{pmatrix}
0&1&1\\
0&0&0\\
0&1&1\\
\end{pmatrix}\right\}$} \text{ and }
\widetilde{\vect{v}_{\sigma,\rho}}=
\begin{pmatrix}
1\\
0\\
1\\
\end{pmatrix}.$$

Now, let $\alpha$ be the root of $x^3-x^2-1=0$ between $1$ and $2$.
If we apply the algorithm from the previous section to this set of matrices with this $\alpha$ we get:
$$X=\left\{
\begin{pmatrix}
1 + \alpha -\alpha^2 \\  3 + 3\alpha -3\alpha^2  \\ 1 + \alpha -\alpha^2  
\end{pmatrix},
\begin{pmatrix}
-1 + \alpha \\  -1 + \alpha \\  -2 + 2\alpha   \\
\end{pmatrix},
\begin{pmatrix}
-1 + \alpha \\  -2 + 2\alpha  \\ -1 + \alpha \\
\end{pmatrix},
\begin{pmatrix}
 -\alpha + \alpha^2 \\   -\alpha + \alpha^2 \\   -\alpha + \alpha^2   \\
\end{pmatrix},
\begin{pmatrix}
1 \\  0  \\ 1 \\
\end{pmatrix}
\right\}$$
It is easy to check that for any $x\in X$ and $M\in \widetilde{A}$, $\frac{1}{\alpha}Mx\in\operatorname{conv}_\le(X)$.
Using Corollary \ref{convexinfhull}, we deduce that the number of induced matchings in a graph of pathwidth 1 and of order $n$
is less than $\alpha^n$.

Moreover, $\alpha$ is the Perron Frobenius eigenvalue of 
$\begin{pmatrix}
0&0&1\\
1&0&0\\
0&1&1\\
\end{pmatrix}$
 which implies that for any  
$\begin{pmatrix}
0\\
1\\
1\\
\end{pmatrix}\cdot
\begin{pmatrix}
0&0&1\\
1&0&0\\
0&1&1\\
\end{pmatrix}^n
\begin{pmatrix}
1\\
0\\
1\\
\end{pmatrix}= \Theta(\alpha^n)$. 
Thus the number of induced matchings of the path of length $n$ is a $\Theta(\alpha^n)$.
In fact, is is easy to deduce from the diagonalization of the matrix that there is a constant $C=\frac{1.31342...}{\alpha}$ (the numerator being the a root of $31x^3-31x^2-12x-1$)
such that this number is $C\alpha^n+o(1)$
We can conclude with the result:
\begin{Proposition}\label{pwinducedmatching}
Let $\alpha$ be the root of $x^3-x^2-1$ between $1$ and $2$, $\alpha\approx1.465571$.
Any graph of pathwidth 1 and of order $n$ admits at most $\alpha^n$ induced matching. 
Moreover, this bound is sharp in the sense that the number of induced matchings of the path of order $n$ is greater than $0.89\alpha^n+o(1)$.
\end{Proposition}
Finally remark, that there is no graph that reach the bound, otherwise $\lambda$ would be of the form $a^{\frac{1}{b}}$ with $a$ and $b$ integers.
That is no possible since the minimal polynomial of $\alpha$ is $x^3-x^2-1$.
Thus there is no finite graph that reaches the bound. However, one could ask whether there is sequence of  increasing graphs such that the number of induced matchings is $\alpha^n+o(1)$.

\paragraph[Perfect total dominating sets]{Perfect total dominating sets:}$\sigma=\{1\}$, $\rho=\{1\}$

We have $p_\sigma=2$, $q_\sigma=1$, $p_\rho=2$ and $q_\rho=1$. 
Thus there should be 6 states, 
however as in the previous example it is clear that two of the states are useless
and thus there should be at most 4 states at the end. 
Running our program give us the following:

$$\widetilde{A}=
\left\{\begin{pmatrix}
1&0&0&0\\
1&1&0&0\\
0&0&0&0\\
0&0&0&0\\
\end{pmatrix},
\begin{pmatrix}
0&0&1&0\\
1&0&0&0\\
0&0&0&1\\
0&1&0&0\\
\end{pmatrix},
\begin{pmatrix}
0&1&0&1\\
0&0&0&0\\
0&1&0&1\\
0&0&0&0\\
\end{pmatrix} \right\}
\text{ and } \widetilde{\vect{v}_{\sigma,\rho}}=
\begin{pmatrix}
1\\
0\\
1\\
0\\
\end{pmatrix}.$$

Remark that there are 3 matrices instead of 4, because one of them was coordinate-wise smaller than another one and was removed from the set.
Let $\alpha$ be the real root of $x^5-4$ and 
$$X=\left\{
\begin{pmatrix} 
0  \\ \frac{1}{4}  \\ 0  \\ 1   \\
\end{pmatrix},
\begin{pmatrix} 
0 \\   \frac{\alpha}{4} \\  0 \\   \frac{3}{4}\alpha    \\ 
\end{pmatrix},
\begin{pmatrix} 
 \frac{\alpha}{4} \\   \alpha \\  0 \\  0    \\ 
\end{pmatrix},
\begin{pmatrix} 
 \frac{\alpha^2}{4} \\  \frac{3}{4}\alpha^2 \\  0 \\  0    \\ 
\end{pmatrix},
\begin{pmatrix} 
 \frac{\alpha^3}{4} \\   \frac{\alpha^3 }{2}\\  0 \\  0    \\ 
\end{pmatrix},
\begin{pmatrix} 
 \frac{\alpha^4}{4} \\  \frac{\alpha^4}{4} \\  0 \\  0   \\ 
\end{pmatrix},
\begin{pmatrix}  
1  \\ 0  \\ 1 \\  0    \\ 
\end{pmatrix}\right\}$$

It is easy to check that for any $x\in X$ and $M\in \widetilde{A}$, $\frac{1}{\alpha}Mx\in\operatorname{conv}_\le(X)$.
Using Corollary \ref{convexinfhull}, we deduce that the number of perfect total dominating sets in a graph of pathwidth 1 and of order $n$
is less than $\alpha^n$.
We can give the following result:
\begin{Proposition}
Let $\alpha$ be the real root of $x^5-4$, $\alpha =4^{\frac{1}{5}}\approx 1.319508$.
Any graph of pathwidth 1 and of order $n$ admits at most $\alpha^n$ perfect total dominating sets. 
Moreover, this bound is reached by the star on 5 vertices.
\end{Proposition}

\paragraph{Other results:}~

We follow with a list of other results that can easily be showed just by running our C++ code.
It is far from being an exhaustive list of what can be done and is mostly an arbitrary choice of parameters to study (the two main criterions were: it is interesting enough for someone to name it and the result is not completely obvious).

Using that the independent dominating sets are the $(\{0\}, \mathbb{N}^+)$-dominating sets we get:
\begin{Proposition}
Any graph of pathwidth 1 and of order $n$ admits at most $2^\frac{n}{2}$ independent dominating sets. 
Moreover, this bound is reached by the path on 2 vertices.
\end{Proposition}
Remark that the independent dominating sets are exactly the maximal independent sets (1-maximal $(\{0\},\mathbb{N})$-dominating sets).

Using that the perfect codes in graphs are the $(\{0\}, \{1\})$-dominating sets, we deduce the following:
\begin{Proposition}
Any graph of pathwidth 1 and of order $n$ admits at most $2^\frac{n}{2}$ perfect codes. 
Moreover, this bound is reached by the path on 2 vertices.
\end{Proposition}

Using that minimal dominating sets in graphs are exactly 1-minimal $(\mathbb{N},\mathbb{N}^+)$-dominating sets, we deduce the following:
\begin{Proposition}
Any graph of pathwidth 1 and of order $n$ admits at most $2^\frac{n}{2}$ minimal dominating sets. 
Moreover, this bound is reached by the path on 2 vertices.
\end{Proposition}

Using that the minimal perfect dominating sets are
the 1-minimal $(\mathbb{N}, \{1\})$-dominating sets we get:
\begin{Proposition}
Let $\alpha$ be the real root of $x^3-x^2-1$ between 1 and 2, $\alpha\approx1.46557$.
Any graph of pathwidth 1 and of order $n$ admits at most $\alpha^n$ minimal perfect dominating sets. 

Moreover, this bound is thigh, since the number of minimal perfect dominating sets of the path of order $n$ is
a $\Theta(\alpha^n)$.
\end{Proposition}
Remark that the second part of this statement comes from the fact that $\alpha$ is the eigenvalue of the matrix that one iterates to build the infinite path.
Also remark that counting all perfect dominating sets is not interesting
since it is $\Theta (2^n)$ on the star of order $n$ (it is also the case for dominating, or total dominating set).

Using that minimal total dominating sets are 
$1$-minimal $(\mathbb{N}^+,\mathbb{N}^+)$-dominating set we get:
\begin{Proposition}
Let $\alpha$ be the real root of $x^3-x-1$ between 1 and 2, $\alpha\approx1.324718$.
Any graph of pathwidth 1 and of order $n$ admits at most $\alpha^n$ minimal total dominating sets. 

Moreover, this bound is thigh, since the number of minimal perfect dominating sets of the path of order $n$ is
a $\Theta(\alpha^n)$.
\end{Proposition}
Our simple algorithm cannot directly find a convex polytope that respects the conditions of
Corollary \ref{convexinfhull} with $\alpha$, but if we add $(2,0,0,0,0,0,2,0,0)^T$ to the set $X$ at the beginning of Algorithm \ref{algopw}
then we find such a set of size $23$ (this vector was found by trial and error helped with some intuition and 2 can be replaced by any greater real number).

Using that maximal strong stable sets are 
$1$-minimal $(\{0\},\{0,1\})$-dominating set we get:
\begin{Proposition}
Any graph of pathwidth 1 and of order $n$ admits at most $3^\frac{n}{3}\approx 1.44225^n$ maximal strong stable sets. 

Moreover, this bound is reached by the path of order 3.
\end{Proposition}

\paragraph{Maximal induced matching}
Maximal induced matchings correspond to $2$-maximal $(\{1\},\mathbb{N})$-dominating sets, 
but it does not seem easy to count them using $1$-maximal $(\sigma,\rho)$-dominating sets.
Unfortunately, we did not describe any way to find a set of matrices corresponding to  $2$-maximal $(\sigma,\rho)$-dominating sets 
(but it is not so hard to generalize to $k$-maximal $(\sigma,\rho)$-dominating sets).
However, we can do it by hand for this particular case.

First we need the following trivial Lemma:
\begin{Lemma}
Let $G(V,E)$ be a graph and $D$ be an induced matching of $G$.
Then $D$ is a maximal induced matching if and only if for all $u,v\in V\setminus D$ and $(u,v)\in E$, $(N(v)\cup N(u))\cap D\not=\emptyset$.
\end{Lemma}

Then, we say that, for any $k$-distinguished graph $g=(G(V,E),S)$, $D$ is a maximal induced matching of $g$ if:
\begin{itemize}
 \item for all $v\in D\setminus S$, $|N(v)\cap D|=1$,
 \item for all $u,v\in V\setminus (D\cup S)$ with $(u,v)\in E$, $(N(v)\cup N(u))\cap D\not=\emptyset$.
\end{itemize}

Then we will have 5 states for the distinguished vertex of the distinguished graphs:
\begin{enumerate}
\setcounter{enumi}{-1}
 \item the vertex is in $D$ and has no neighbor in $D$,
 \item the vertex is in $D$ and has exactly one neighbor in $D$,
  \item the vertex is not in $D$, has no neighbor in $D$ and has no neighbor in $V\setminus S$ that has no neighbor in $D$,
 \item the vertex is not in $D$, has no neighbor in $D$ and has at least one neighbor in $V\setminus S$ that has no neighbor in $D$,
 \item the vertex is not in $D$ and has at least one neighbor in $D$.
\end{enumerate}
For any $1$-distinguished graph $g=(G(V,E),\{s\})$ let $\Psi(g)$ be the vector whose $i$-th coordinate
gives the number of maximal induced matchings of $g$ where  $s$ is in the $i$th state. 
\begin{Lemma}
Let $g=(G(V,E),\{o\})$ and  $g'=(G'(V',E'),\{n\})$  be two distinguished graphs, where $g'$ is the graph obtained by extending $g$ and adding $n$
and an edge between $o$ and $n$.
Then: $$\Psi(g')=
\begin{pmatrix}
0&0&1&1&1\\
1&0&0&0&0\\
0&0&0&0&1\\
0&0&1&0&0\\
0&1&0&0&0\\
\end{pmatrix}
\Psi(g)$$
\end{Lemma}
\begin{proof}
For any maximal induced matching $D$ of $g$ where $o$ is in state 2, 3 or 4, 
then $D'=D\cup n$ is a maximal induced matching of $g'$ where $n$ is in state $0$,
since then $s$ has a neighbor in $D'$ and $n$ has no neighbor in $D'$.

For any maximal induced matching $D$ of $g$ where $o$ is in state $3$,
$D$ is not a maximal induced matching of $g'$ since $o$ has no neighbor in $D$ 
and has a neighbor that has no neighbor in $D$.

For any maximal induced matching $D$ of $g$ where $o$ is in state $2$,
$D$ is a maximal induced matching of $g'$ where $n$ is in state $3$ since,
$v$ is the only neighbor of $o$.

For any maximal induced matching $D$ of $g$ where $o$ is in state $4$,
$D$ is a maximal induced matching of $g'$ where $n$ is in state $2$ since,
$v$ is the only neighbor of $o$.

For any maximal induced matching $D$ of $g$ where $o$ is in state $1$
$D\cup\{n\}$ is not an induced matching of $g'$ since $o$ has 2 neighbors in $D\cup \{n\}$.

For any maximal induced matching $D$ of $g$ where $o$ is in state $1$
$D$ is a maximal induced matching of $g'$ where $n$ is in state $4$ since it shares an edge with $v$.

For any maximal induced matching $D$ of $g$ where $o$ is in state $0$
$D$ is not a maximal induced matching of $g'$ since $o$ has 0 neighbors in $D$.

For any maximal induced matching $D$ of $g$ where $o$ is in state $0$
$D\cup\{n\}$ is a maximal induced matching of $g'$ where $n$ is in state $1$ since it shares an edge with $v$.

It covers all the possible cases, and it gives the matrix from the Theorem.
\end{proof}

A similar case analysis give the matrices of the other possible extensions and we get the following set of matrices:
$$\resizebox{0.9\hsize}{!}{ $A=\left\{\begin{pmatrix}
0&0&1&1&1\\
1&0&0&0&0\\
0&0&0&0&1\\
0&0&1&0&0\\
0&1&0&0&0\\
\end{pmatrix},\begin{pmatrix}
1&0&0&0&0\\
1&1&0&0&0\\
0&0&0&0&0\\
0&0&1&1&0\\
0&0&0&0&1\\
\end{pmatrix},
 \begin{pmatrix}
0&1&1&0&1\\
0&0&0&0&0\\
0&1&1&0&1\\
0&0&0&0&0\\
0&0&0&0&0\\
\end{pmatrix}, \begin{pmatrix}
1&0&0&0&0\\
0&1&0&0&0\\
0&0&1&0&0\\
0&0&0&1&0\\
0&0&0&0&1\\
\end{pmatrix} \right\}$}$$
 
We can also show the following Lemma:
\begin{Lemma}
Let $g=(G(V,E),\{s\})$ be a $1$-distinguished graph and $G'$ 
be the only graph that can be obtained by completing $g$.
Then the number of maximal induced matchings of $G'$ is given by  
$(0,1,1,0,1)^T\cdot \Psi(g)$.
\end{Lemma}

Thus the maximal number of maximal induced matchings
of a graph of pathwidth at most 1 is given by:
$$\resizebox{0.9\hsize}{!}{ $\#_{\text{maximal induced matching}}(n) =\max \left\{ 
\begin{pmatrix}
0\\
1\\
1\\
0\\
1\\
\end{pmatrix} \cdot \left(\prod_{i=1}^{n-1}M_i\right) 
\begin{pmatrix}
1\\
0\\
1\\
0\\
0\\
\end{pmatrix}
: (M_i)_{1\le i\le n-1}\in A^{n-1}\right\}.$}$$

Running the algorithm previously described with this set of matrices and
$\alpha =13^\frac{1}{9}$ gives a set $X$ of $25$ vectors that respects 
the conditions of Corollary \ref{convexinfhull}. Thus we deduce the following result:
\begin{Proposition}
Any graph of pathwidth 1 and of order $n$ admits at most $13^\frac{n}{9} \approx 1.32975^n$ 
maximal induced matching.
Moreover the bound is reached by the following graph:
\begin{figure}[H]
\centering
\includegraphics{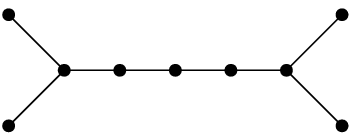}
\end{figure}
\end{Proposition}

\subsubsection{Results and examples in pathwidth 2}
For graphs of pathwidth $2$, 
we need to use $2$-distinguished graphs so there are $12$ matrices.
Moreover, the dimensions of these matrices is at least $9$ for non-trivial $\sigma$ and $\rho$.
So we really need the computer and there is no simple example to do by hand.

Using that dominating stable sets are $(\{0\}, \mathbb{N}^+)$-dominating sets we get:
\begin{Proposition}
Any graph of pathwidth $2$ and of order $n$ contains at most $3^\frac{n}{3}\approx 1.44225^n$ dominating stable sets.
The bound is reached by the triangle.
\end{Proposition}
Remark that this result is a trivial case of 
the famous result of Moon and Moser stating that there are at most $3^\frac{n}{3}\approx 1.44225^n$ 
dominating stable sets in any graph of order $n$.
The two next Lemmas might also be simple corollaries of this result, 
but if it is the case it is less obvious.

Using that perfect total dominating sets are $(\{1\},\{1\})$-dominating set, we get:
\begin{Proposition}
Any graph of pathwidth $2$ and order $n$ contains at most $3^\frac{n}{3}\approx 1.44225^n$ perfect total dominating sets.
The bound is reached by the triangle.
\end{Proposition}

Using that perfect codes are $(\{0\}, \{1\})$-dominating sets we get:
\begin{Proposition}
Any graph of pathwidth $2$ and order $n$ contains at most $3^\frac{n}{3}\approx 1.44225^n$ perfect codes.
The bound is reached by the triangle.
\end{Proposition}

Using that induced matchings are $(\{1\},\mathbb{N})$-dominating sets, we get:
\begin{Proposition}
Any graph of pathwidth $2$ and order $n$ contains at most $4^\frac{n}{3}\approx 1.58740^n$  induced matchings.
The bound is reached by the triangle.
\end{Proposition}

\paragraph{Maximal induced matching}
As already mentionned, maximal induced matchings correspond to $2$-maximal $(\{1\},\mathbb{N})$-dominating sets, 
and we did not explain how to count them. 
However, we can use exactly the same set of states than for the graphs of pathwidth $1$ (see paragraph of the same name in Section \ref{examplepw1}).
We do not describe the details but they are implemented in C++.
We can deduce the following Theorem:
\begin{Lemma}
Any graph of pathwidth $2$ and order $n$ contains at most $5^\frac{n}{4}\approx 1.49535^n$ maximal induced matchings.
Moreover the bound is reached by the following graph:
\begin{figure}[H]
\centering
\includegraphics{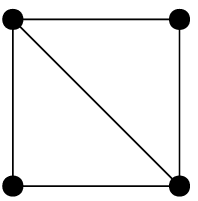}
\end{figure}
\end{Lemma}
The set $X$ that our program returns contains 386 vectors and it takes more than 5 hours to find it on a small laptop. 

\paragraph{Minimal dominating set}
Minimal dominating sets are probably among the most interesting minimal $(\Sigma,\rho)$-dominating sets to study.
Unfortunately, Algorithm \ref{algopw} does not seem to find the set $X$ needed for Corollary \ref{convexinfhull}
and we were not able to show a sharp bound on the groth rate by directly applying our method
We explain in Section \ref{sectionjspr} that we could get good approximation of the upper bound anyway.
However, with some extra work we can still obtain a sharp bound.

We can define states that are specific to this problem (we have 37 states instead of 165) and we are able to find a set $X$ that respects the conditions of Corollary \ref{convexinfhull} 
if we initialize Algorithm \ref{algopw} with a larger set of vectors.
One of the main improvement is that instead of choosing the certificates from the start and checking that they are indeed certificates, we choose them only when we need to.
Something similar could be done in general, but we need to introduce new states for $S$ where we do not know exactly the states of the vertices 
(something similar to the 37th state: ``the two nodes are dominated and at least one of them is a certificate''). 
In general, it might not be helpfull to introduce these new states, but it definitely makes things easier in this particular case.

We get the following result:
\begin{Proposition}
 The number of minimal dominating sets of a graph of pathwidth 2 of order $n$ is at most $6^{\frac{n}{4}}\approx 1.56508^n$.
 This bounds is sharp and is reached by any collection of cycles of order $4$.
\end{Proposition}
\begin{proof}
Given a 2-distinguished graph  $g=(G(V,E),S,<_{S})$ and a minimal-dominating set $D$ of $g$ for $v\in S$ the state of $v$ can be:
\begin{description}[align=left]
\item [D] if $v\in D$ and has a neighbor that is a certificate,
\item [S] if $v\in D$ and has no neighbor in $D$,
\item [L] if $v\in D$ and has a neighbor in $D$,
\item [P] if $v\not\in D$ and is the certificate of an element of $D\cap S$,
\item [d] if $v\not\in D$ and has a neighbor in $D$,
\item [F] if $v\not\in D$ has no neighbor in $D$.
\end{description}
The state of $S$ is either the product of the states of the two elements of $S$ (that is $S=\{u,v\}$ with $u<v$ is in state $6i+j$ iff $u$ is in state $i$ and $v$ is in state $j$) or the special state $\pi$.
If $S$ is in state $\pi$ then at least one element of $S$ is the certificate of an element of $D\cap S$.

This states are such that given for any state of $S$ the number of minimal dominating sets of a 2-distinguished graph $g=(G(V,E),S,<_{S})$, 
one can easily compute for any state of $S'$ the number of minimal dominating sets of a 2-distinguished graph $g'=(G'(V',E'),S',<_{S'})$ obtained by extending $g$.
Moreover, one can once again compute a set of matrices that correspond to the different possible ways to extend a 2-distinguished graph.
The details are left to the interested reader (however the matrices can be obtained from the C++ code).

Now that we have our set of matrices we only need a set $X$ that respects the conditions of Corollary \ref{convexinfhull} with $\alpha=6^{\frac{1}{4}}$.
However, once again Algorithm \ref{algopw} does not seem to terminate in this case.
The solution is to apply Alorithm \ref{algopw} starting with some other vertices. Let 
\begin{align*}
 X_0=&\{e_{1}, e_{4}, e_{5}, e_{10}, e_{15}, e_{16}, e_{17}, e_{18}, e_{19}, e_{20}, e_{22}, e_{23}, e_{24},\\
 &e_{25}, e_{26}, e_{27}, e_{28}, e_{29}, e_{30}, e_{33}, e_{34}, e_{35}, e_{36}\},
\end{align*}
where $e_i$ is the vector from $\mathbb{R}^{37}$ whose $i$th coordinate is $4$ and the others are $0$ and let $v_0$ be the vector corresponding to $g_0=(G(S,\emptyset),S)$.
Then running Algorithm \ref{algopw} with $X$ initilizalied at $X_0\cup\{v_0\}$ gives a set $X$ of size $131$ that respects the conditions of Corollary \ref{convexinfhull} with $\alpha=6^{\frac{1}{4}}$.
This concludes the proof.
\end{proof}
We provide with our C++ implementation a file \footnote{\texttt{dom\_pw2\_in}} that contains the set $X_0$, to avoid the tedious task of typing $23$ vectors in a terminal.
Remark, that it also works with $X'_0=\{e_i|i\in\{1,\ldots,37\}\}$, but the $X_0$ from the proof gives a smaller final set $X$. 
The set $X_0$ was constructed by removing vertices from $X'_0$ by trial and error (and it is probably not even the ``best'' subset of $X'_0$).

\subsubsection{Joint spectral radius}\label{sectionjspr}
As already mentioned the algorithm that computes the set $X$ does not necessarily terminates. 
However, using the notion of joint spectral radius we can still computes bounds in this case.

For any given set of matrices $A$ and any  sub-multiplicative matrix norm $||.||$, the quantity $p(A)=\lim_{n\rightarrow\infty}\max\{||M_1\ldots M_n||^{\frac{1}{k}}: M_i\in A\}$
is the \emph{joint spectral radius of the set $A$} \cite{blondel_jungers_2010}. 
Remark that $p(A)$ does not depend on the chosen  sub-multiplicative matrix norm.

In fact, given a set of matrices $A\in \mathbb{R}^{+n\times n}$, a vector $\vect{v}_{\sigma,\rho}\in \mathbb{R}^{+n}$ and a  set of vectors $P\in \mathbb{R}^{+n}$ such that all the coordinates are accessible and co-accessible then  $p(A)$ is the smallest real such that:
$$\lim_{m\rightarrow \infty}\left(\max \left\{ \vect{p}\cdot \left(\prod_{i=1}^{m-k}M_i\right) \vect{v}_{\sigma,\rho}:\vect{p}\in P , (M_i)_{1\le i\le m-k}\in A^{m-k}\right\}\right)^{1/m} = p(A) $$

The  joint spectral radius generalizes the notion of spectral radius to set of matrices, but we lose a lot of nice properties of the spectral radius.
In particular, the spectral radius is easy to express as a root of a polynomial, but the joint spectral radius is much harder to compute.
However, there are approximation algorithms for the joint spectral radius that run in exponential time \cite{blondel_jungers_2010}. 
Thus we can always obtain arbitrarily good bounds for the number of $(\sigma, \rho)$-dominating sets in bounded pathwidth.

\subsection{Forests and trees}
Let $\sigma$ and $\rho$ be two recognizable sets of positive integers and  $\#_{(\sigma,\rho)}(n)$ be the maximal number of (all, 1-minimal, 1-maximal) $(\sigma,\rho)$-dominating sets in forests (or trees) of order $n$.
Then we can use Lemma \ref{expressiontocountalltrees}(or Lemmas \ref{expressiontocountmintree}, \ref{expressiontocountmaxtree}, \ref{expressiontocountallforests}, \ref{expressiontocountminforest}, \ref{expressiontocountmaxforest})
to find an integer $m$, two vectors $\vect{v},\vect{p}\in \mathbb{N}^{m}$ and a
set of bilinear maps $A\subseteq (\mathbb{N}^m)^{\mathbb{N}^m\times \mathbb{N}^m}$ (a singleton for trees and a pair for forests) such that:

If  $F\subseteq (\mathbb{N}^{n}\times\mathbb{N})$ is the smallest set such that
\begin{itemize}
 \item $(\vect{v} , 1)\in F$,
 \item for any $(\vect{w},i),(\vect{u},j)\in F$ and $\Phi\in A$, $(\Phi(\vect{w},\vect{u}),i+j)\in F$.
\end{itemize}
then there is a tree (or a forest) of order $n$ that admits $N$ (1-minimal, 1-maximal) $(\sigma,\rho)$-dominating sets if and only if
there exists $\vect{u}$ such that $(\vect{u},n)\in F$ and $N= \vect{p}\cdot \vect{u}$.
Thus in particular $\#_{\sigma,\rho}(n) =\max \{\vect{p}\cdot \vect{u} : (\vect{u},n)\in F\}$.
In this subsection we explain how we compute this quantity.

We say that $i$-th coordinate is \emph{accessible} if there is $(\vect{u},k)\in F$ such that the $i$-th coordinate of $\vect{u}$ is non zero.
Let $\vect{e_i}$ be the vector whose $i$th coordinate is $1$ and the others are $0$. 
We define inductively the set of \emph{co-accessible} coordinates:
\begin{itemize}
 \item if the $i$-th coordinate of $\vect{p}$ is non zero then $i$ is co-accessible,
 \item if for some $\phi\in A$ one of the co-accessible coordinate of $\Phi(\vect{e_i}, \vect{e_j})$ is non zero and the $j$-th coordinate is accessible then 
 the $i$-th coordinate is co-accessible,
 \item if for some $\phi\in A$ one of the co-accessible coordinate of $\Phi(\vect{e_i}, \vect{e_j})$ is non zero and the $i$-th coordinate is accessible then 
 the $j$-th coordinate is co-accessible.
\end{itemize}

Note that the sets of accessible and co-accessible coordinates can easily be computed by a recursive algorithm.
The motivation of these definitions is that we can ignore coordinates that are not accessible and co-accessible since they do not influence the result.

Let $\widetilde{m}$ be the number of accessible and co-accessible coordinates.
Let $h$ be the endomorphism that maps the $i$-th coordinate to the $i$-th  accessible and co-accessible coordinate and $h^T$ 
be the endomorphism that maps the $i$-th accessible and co-accessible coordinate to the $i$-th coordinate. 
Let $\widetilde{A}=\{h\circ\Phi\circ(h^T\times h^T): \Phi\in A\}$, $\widetilde{\vect{v}}=h(\vect{v})$ and $\widetilde{\vect{p}}=h(\vect{p})$.
Finally let $\widetilde{F}\subseteq (\mathbb{N}^{n}\times\mathbb{N})$ be the smallest set such that:
\begin{itemize}
 \item $(\vect{v} , 1)\in \widetilde{F}$,
 \item for any $(\vect{w},i),(\vect{u},j)\in \widetilde{F}$ and $\Phi\in \widetilde{A}$, $(\Phi(\vect{w},\vect{u}),i+j)\in \widetilde{F}$.
\end{itemize}
We get:
$$\#_{\sigma,\rho}(n) =\max \{\widetilde{\vect{p}}.\vect{u} : (\vect{u},n)\in \widetilde{F}\}\,.$$

We can finally explain how to compute this quantity.
For any set of points $X$, we denote by $\operatorname{conv}(X)$ the convex hull of $X$.
We can now give the following Theorem:
\begin{Theorem}\label{convexhullTree}
Let $\alpha$ be a positive real, $\widetilde{A}\subseteq (\mathbb{N}^m)^{\mathbb{N}^m\times \mathbb{N}^m}$, $\widetilde{\vect{v}}\in\mathbb{R^+}^m$ and $\widetilde{F}$ be the smallest set such that:
\begin{itemize}
 \item $(\widetilde{\vect{v}} , 1)\in \widetilde{F}$,
 \item for any $(\vect{w},i),(\vect{u},j)\in \widetilde{F}$ and $\Phi\in \widetilde{A}$, $(\Phi(\vect{w},\vect{u}),i+j)\in \widetilde{F}$.
\end{itemize}

If there is a bounded set of vectors $X\subseteq\mathbb{R}^m$  such that:
\begin{itemize}
 \item $\frac{\widetilde{\vect{v}}}{\alpha}\in\operatorname{conv}(X)$,
 \item $\forall \vect{x},\vect{x'}\in X$ and for any $\Phi\in \widetilde{A}$, $ \Phi(\vect{x},\vect{x'})\in\operatorname{conv}(X)$.
\end{itemize}
Then for any $\vect{p}\in\mathbb{R^+}^m$ and for all integers $n$
$$\max \left\{ \left|\vect{p}\cdot u\right|:(u,n)\in \widetilde{F}\right\}\le \alpha^n \max_{ x\in X}|\vect{p}\cdot x|$$
\end{Theorem}
\begin{proof}
The proof is mostly trivial manipulations on convex sets.
First, note that for any $\Phi\in A$, $\operatorname{conv}(\{\Phi(x,y): x,y\in X\})\subseteq\operatorname{conv}(X)$.

Let us show by induction on the elements of $F$ that for all $(\vect{u},n)\in \widetilde{F}$,
$\frac{\vect{u}}{\alpha^n}\in  \operatorname{conv}(X)$.
For any $(\vect{u},n)\in F$ there are two possibilities:
\begin{itemize}
 \item $(\vect{u},n)=(\widetilde{\vect{v}},1)$, then by definition of $ \widetilde{F}$, $\frac{\widetilde{\vect{v}}}{\alpha}\in\operatorname{conv}(X)$.
 \item $\vect{u}=\Phi(\vect{u_1},\vect{u_2})$ with $(\vect{u_1},i)\in  \widetilde{F}$, $(\vect{u_2},j)\in  \widetilde{F}$, $i+j=n$ and $\Phi\in \widetilde{A}$.
 By induction hypothesis there are two functions $f_1, f_2:X\mapsto[0,1]$ such that:
 $\sum_{x\in X} f_1(x)=\sum_{x\in X} f_2(x)=1$, $\frac{\vect{u_1}}{\alpha^{i}}=\sum_{x\in X} f_1(x) x$ and $\frac{\vect{u_2}}{\alpha^{j}}=\sum_{x\in X} f_2(x) x$.
 By bilinearity of $\Phi$ we get:
$$  \frac{\vect{u}}{\alpha^n}=\frac{\Phi(\vect{u_1},\vect{u_2})}{\alpha^n}
  =\Phi\left(\frac{\vect{u_1}}{\alpha^i}, \frac{\vect{u_2}}{\alpha^j}\right)
  =\sum_{x\in X}\sum_{y\in X}f_1(x)f_2(y)\Phi(x,y)$$
  Moreover, $\sum_{x\in X}\sum_{y\in X}f_1(x)f_2(y)=1$ implies that $\frac{\vect{u}}{\alpha^n}\in\operatorname{conv}(\{\Phi(x,y): x,y\in X\})\subseteq\operatorname{conv}(X)$.  
\end{itemize}

Now we know that for all $(\vect{u},n)\in \widetilde{F}$,
$\frac{\vect{u}}{\alpha^n}\in  \operatorname{conv}(X)$. 
It implies that there is a function $f:X\mapsto[0,1]$ such that:
 $\sum_{x\in X} f(x)=1$ and $\frac{\vect{u}}{\alpha^{n}}=\sum_{x\in X} f(x) x$.
Thus $|\vect{p}\cdot\frac{\vect{u}}{\alpha^{n}}|=|\sum_{x\in X} f(x) \vect{p}\cdot x| \le \max_{ x\in X}|\vect{p}\cdot x|$.
We finally get $$|\vect{p}\cdot \vect{u}|\le \alpha^n \max_{ x\in X}|\vect{p}\cdot x|\,.$$
This concludes the proof.
\end{proof}

This Theorem is an analogue to Theorem \ref{convexhull} and we can also obtain the analogue to Corollary \ref{convexinfhull}:
\begin{Corollary}\label{convexinfhullTree}
Let $\alpha$ be a positive real, $\widetilde{A}\subseteq (\mathbb{N}^m)^{\mathbb{N}^m\times \mathbb{N}^m}$, $\widetilde{\vect{v}}\in\mathbb{R^+}^m$ and $\widetilde{F}$ be the smallest set such that:
\begin{itemize}
 \item $(\widetilde{\vect{v}} , 1)\in \widetilde{F}$,
 \item for any $(\vect{w},i),(\vect{u},j)\in \widetilde{F}$ and $\Phi\in \widetilde{A}$, $(\Phi(\vect{w},\vect{u}),i+j)\in \widetilde{F}$.
\end{itemize}

If there is a bounded set of vectors $X\subseteq\mathbb{R}^m$  such that:
\begin{itemize}
 \item $\frac{\widetilde{\vect{v}}}{\alpha}\in\operatorname{conv}_\le(X)$,
 \item $\forall \vect{x},\vect{x'}\in X$ and for any $\Phi\in \widetilde{A}$, $ \Phi(\vect{x},\vect{x'})\in\operatorname{conv}_\le(X)$.
\end{itemize}
Then for any $\vect{p}\in\mathbb{R^+}^m$ and for all integers $n$
$$\max \left\{ \left|\vect{p}\cdot u\right|:(u,n)\in \widetilde{F}\right\}\le \alpha^n \max_{ x\in X}|\vect{p}\cdot x|$$
\end{Corollary}

Moreover, in the case of forest, we can once again use Lemma \ref{feketesApplication} to replace the constant by 1.
However, in the case of trees this lemma cannot be applied since $\#_{\sigma,\rho}(n)$ is not necessarily super-multiplicative (take $\sigma=\rho=\{0\}$ then $\#_{\sigma,\rho}(1)=2$, but $\#_{\sigma,\rho}(n)=1$ for $n>1$).

Once again there is not always such a set $X$, but in practice we can often find one with the simple following algorithm:

\begin{algorithm}[H]\label{findconvsettree}
 \KwData{$\widetilde{A}, \widetilde{\vect{v}}$, $\alpha$}
 \KwResult{A set $X$ such as in Corollary \ref{convexinfhull}}
 $X:=\{\frac{\widetilde{\vect{v}}}{\alpha}\}$\;
  \While{$\{\Phi(\vect{x},\vect{y}): \vect{x},\vect{y}\in X, \Phi\in \widetilde{A}\}\not\subseteq \operatorname{conv}_\le(X)$}{
    $X:=\operatorname{Hull}_\le(\{\Phi(\vect{x},\vect{y}): \vect{x},\vect{y}\in X, \Phi\in \widetilde{A}\}\cup X)$\;

 }
\caption{Computation of the set $X$}
\end{algorithm}
\subsubsection{Results and examples}
Rote used a similar technique in \cite{r-mnmds-19,arxivRote} to count and bound the number of minimal dominating sets in trees.
He showed that the number of minimal dominating sets of a tree (or a forest) of order $n$ is bounded by $95^{\frac{n}{13}}$ and that it is sharp.
Our technique can be used to reprove this result.
We give some examples of application of the technique:
\paragraph{Independent dominating sets}
Independent dominating sets are exactly $(\{0\},\{\mathbb{N}^+\})$-dominating sets.
Applying the method previously described, after deletion of the useless coordinates we obtain vectors of dimension 3.
The 3 states correspond to:
\begin{enumerate}
\setcounter{enumi}{-1}
 \item $v$ is in $D$ and has no neighbor in $D$,
 \item $v$ is not in $D$ and has no neighbor in $D$,
 \item $v$ is not in $D$ and has at least one neighbor in $D$.
\end{enumerate}

The bilinear map $\widetilde{\Phi}_{(\{0\},\{\mathbb{N}^+\})}$ that corresponds to the composition of two 1-distinguished graphs is such that for any $u,v\in\mathbb{R}^3$:
$$\widetilde{\Phi}_{(\{0\},\{\mathbb{N}^+\})}(u,v)= 
\begin{pmatrix}
u_0(v_1+v_2)\\
u_1 v_2\\
u_1v_0+u_2(v_0+v_2)
\end{pmatrix}
$$
The bilinear map $\widetilde{\Delta}_{(\{0\},\{\mathbb{N}^+\})}$ that corresponds to the union of two 1-distinguished graphs is such that for any $u,v\in\mathbb{R}^3$:
$$\widetilde{\Delta}_{(\{0\},\{\mathbb{N}^+\})}(u,v)= 
\begin{pmatrix}
u_0(v_0+v_2)\\
u_1 (v_0+v_2)\\
u_2(v_0+v_2)
\end{pmatrix}
$$
The vector $\vect{v}_{(\{0\},\{\mathbb{N}^+\})}=\Psi_{(\{0\},\{\mathbb{N}^+\})}(G(\{x\},\emptyset), \{x\})$ is:
$\vect{v}_{(\{0\},\{\mathbb{N}^+\})}=
\begin{pmatrix}
1\\
1\\
0
\end{pmatrix}.$

Let $\alpha=\sqrt{2}$ and $X=\left\{
\begin{pmatrix}
0\\
0\\
\frac{\alpha}{2}
\end{pmatrix},
\begin{pmatrix}
\frac{1}{2}\\
0\\
\frac{1}{2}
\end{pmatrix},
\begin{pmatrix}
\frac{\alpha}{2}\\
\frac{\alpha}{2}\\
0
\end{pmatrix}
\right\}$.

It is clear that $\frac{\vect{v}_{(\{0\},\{\mathbb{N}^+\})}}{\alpha}\in\operatorname{conv}_\le(X)$ and it is easy to check that for all:
$u,v\in X$, $\widetilde{\Delta}_{(\{0\},\{\mathbb{N}^+\})}(u,v)\in \operatorname{conv}_\le(X)$ and $\widetilde{\Phi}_{(\{0\},\{\mathbb{N}^+\})}(u,v)\in \operatorname{conv}_\le(X)$.
By Corollary \ref{convexinfhullTree}, the number of Independent dominating sets in a forest of order $n$ is at most $\sqrt{2}^n$.
This bound is reached for graphs made of copies of the complete graph of order 2.

Note that Algorithm \ref{findconvsettree} does not find the set $X$, but find a sequence of set that converges toward $X$.

In fact, the bound is also sharp for trees because of the following example (see fig. \ref{indepdomtree}).
For any positive integer $n$, let $G_n(\{s,s_1,\ldots,s_{2n}\}, \{(s,s_{2i-1}): i\in[1,n]\}\cup\{(s_{2i-1}, s_{2i}): i\in[1,n]\})$.
Let $D$ be a set such that for all $i$, $|S\cap\{s_{2i},s_{2i+1}\}|=1$ and $s\in D$ iff for all $i$, $s_{2i-1}\in S$. 
Then $D$ is an independent dominating set of $G_n$. Thus there are at least $2^{n}$ independent dominating sets of $G_n$ which is of order $2n+1$.

\begin{figure}[H]
\centering
\includegraphics{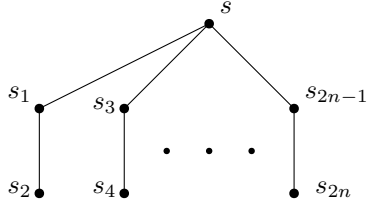}
\caption{Trees with $\Theta(2^{\frac{n}{2}})$ independent dominating sets \label{indepdomtree}}
\end{figure}
\begin{Proposition}
 The number of independent dominating sets in forest and trees is at most $2^{\frac{n}{2}}$ where $n$ is the order of the graph.
It is sharp in the sens that the number of independent dominating sets is at least $\frac{ 2^{\frac{n}{2}}}{2}$ for infinitely many trees.
\end{Proposition}

Composition of a $1$-distinguished graph $g$ with $p_2=(G(\{1,2\},\{(1,2)\}),1)$ gives:
$$\widetilde{\Phi}_{(\{0\},\{\mathbb{N}^+\})}(\Psi(g), \Psi(p_2))= 
\begin{pmatrix}
1&0&0\\
0&1&0\\
0&1&2
\end{pmatrix}
\Psi(g).$$
Composition with $p_2$ correspond to adding 2 vertices  and the Perron-Frobenius eigenvalue of this matrix is $2$ while the other two eigenvalues are both $1$.
It gives a less combinatorial proof that there are trees with $\Theta(2^{\frac{n}{2}})$ independent dominating sets.
It is not a coincidence since by starting with a single vertex and composing iteratively with $p_2$ one gets the family of graphs $G_n$.

\paragraph{Induced matching}
We know from Proposition \ref{pwinducedmatching} that the number of induced matchings of the path of length $n$ is a $\theta(\alpha^n)$ where $\alpha$ is the real
root of $x^3-x^2-1$ between 1 and 2.
Using the fact that induced matchings are $(\{1\},\mathbb{N})$ dominating sets we can apply our technique and obtain:
\begin{Proposition}
 Let $\alpha$ be the real root of $x^3-x^2-1$ between 1 and 2, $\alpha\approx1.46557$. 
 Then the number of induced matchings in a forest of order $n$ is bounded by $\alpha^n$.
 Moreover, this value of $\alpha$ is sharp even for paths.
\end{Proposition}

\paragraph{Total perfect dominating set}
Total perfect dominating sets are exactly $(\{1\},\{1\})$-dominating sets.
We can compute the two bilinear maps corresponding to union and composition of $1$-distinguished graph and they are given by:
$$ \widetilde{\Phi}_{(\{1\},\{1\})}(u,v)= 
\begin{pmatrix}
u_0 v_2\\
u_0 v_0+u_1 v_2\\
u_2(v_3)\\
u_2 v_1+u_3v_3
\end{pmatrix},
\widetilde{\Delta}_{(\{1\},\{1\})}(u,v)= 
\begin{pmatrix}
u_0(v_1+v_3)\\
u_1(v_1+v_3)\\
u_2(v_1+v_3)\\
u_3(v_1+v_3)\\
\end{pmatrix}
$$

Using Algorithm \ref{findconvsettree} with $\alpha = 4^\frac{1}{5}$, we can  compute a set $X$ of points that respects the conditions of Corollary \ref{convexinfhullTree}.
It gives the following result:
\begin{Proposition}
The number of total perfect dominating sets of a forest of order $n$ is upper-bounded by $4^\frac{n}{5}$.
This number is reached by  disjoint union of stars on 5 vertices.
\end{Proposition}

This bound is also true, but not sharp for trees:
\begin{Proposition}\label{treetpd}
Let $\alpha = (2^{27}\times 7)^\frac{1}{85}\approx1.275157$.
There exists a positive constant $C$ such that the number of total perfect dominating sets of a tree of order $n$ is upper-bounded by $C\alpha^n$.
This value of $\alpha$ is sharp.
\end{Proposition}
\begin{proof}
The proof that this bound is correct can be done by finding a set $X$ that respects the conditions of Corollary \ref{convexinfhullTree}.
it can be done using Algorithm \ref{findconvsettree}.
However, since the computations are terribly long (because of the algebraic number of high degree) we give the set $X$ in Annex \ref{Xperfecttotaldomtree}.

Let $G$ be the graph depicted in figure \ref{treetpds}.
\begin{figure}[H]
\centering
\includegraphics[scale=0.85]{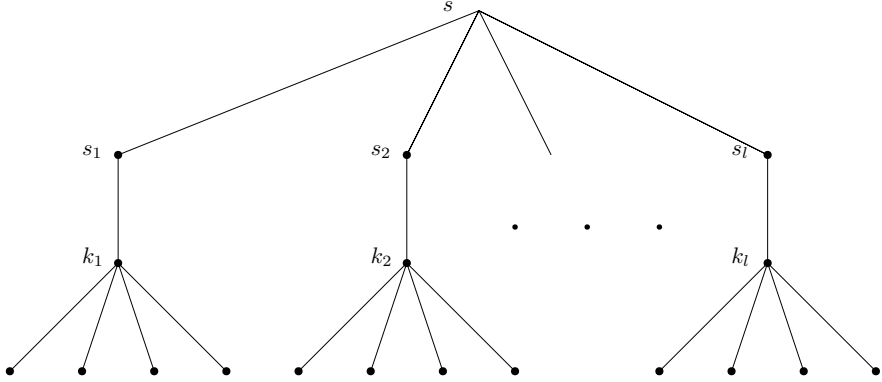}
\caption{The graph $G$ used in the construction of trees with $\alpha^n$ total perfect dominating sets \label{treetpds}}
\end{figure}

Let $D$ be a perfect total dominating set of this graph. 
Each leaf of the tree needs a neighbor in $D$, so all the $k_i$ are in $D$.
This implies that $s$ is not in $D$, otherwise the $s_i$ would have two neighbors in $D$.
Moreover exactly one of the $s_i$ must be in $D$, and then $k_i$ is the only other vertex in $D$ in the subtree rooted in $k_i$.
For the other subtrees, exactly one of the leaf of each of them is in $D$.
Thus there are $(4^{l-1}l)$  perfect total dominating sets.

Now let $G'$ be a chain of $k$ copies of $G$ where two consecutive copies share an edge between their roots.
Then since the roots cannot be in the perfect total dominating sets the number of perfect total dominating sets is exactly $(4^{l-1}l)^k$.
For $l = 14$ it gives $(2^{27}\times 7)^k$ perfect total dominating sets.
In this case the graph $G'$ has $n=k (6\times14+1)=85k$ vertices, and $\alpha^n$ total perfect dominating sets.
\end{proof}

\paragraph{Perfect code}
Using that the perfect codes in graphs are exactly $(\{0\},\{1\})$-dominating sets, we deduce the following:
\begin{Proposition}
Any forest of order $n$ admits at most $2^\frac{n}{2}$ perfect codes. 
Moreover, this bound is sharp for unions of paths of order 2.
\end{Proposition}

Once again, we can get a smaller bound for trees.
\begin{Proposition}
Let $\alpha = 3^{\frac{1}{7}}\approx1.16993$.
There exists a positive constant $C$ such that the number of perfect codes of a tree of order $n$ is upper-bounded by $C\alpha^n$.
This value of $\alpha$ is tight.
\end{Proposition}
\begin{proof}
The proof that this bound is correct can be done by finding a set $X$ that respects the conditions of Corollary \ref{convexinfhullTree}.
it can be done using Algorithm \ref{findconvsettree}.

Let $G$ be the graph depicted in figure \ref{treepc}.
\begin{figure}[H]
\centering
\includegraphics{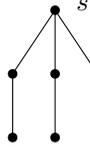}
\caption{The graph $G$ used in the construction of trees with  $7^{\frac{n}{3}}$ perfect codes \label{treepc}}
\end{figure}
The number of perfect codes of $G$ is $3$ and none of the perfect codes contain $s$.
The tree obtained by connecting multiple copies of $G$ with a path going through the copies of $s$ has at least (and in fact exactly) $7^{\frac{n}{3}}$ perfect codes.
\end{proof}

\section{Conclusion}
We explained here a technique to bound the number of (1-minimal, 1-maximal) $(\sigma,\rho)$-dominating sets of graphs of bounded pathwidth, trees or forests when $\sigma$ and $\rho$ are recognizable.
The technique could be adapted to find lower-bounds (in the case of pathwidth, one can compute the joint spectral subradius of the same set of matrices),
but it doesn't seem to be interesting in general.
The main technique presented here probably cannot work all the time to find the minimal upper-bound. 

However, in the case of bounded pathwidth, the computation of these bounds is equivalent to the computation of the joint spectral radius of some particular finite set of computable matrices.
It is known that the joint spectral radius can be approximated, which implies that we can approximate the optimal bound.
We do not have any similar result for trees and forests (and as far as the author knows there is no notion that ``generalizes'' the joint spectral radius to bilinear maps).
Thus the first question to solve would be:

\begin{Problem}
 Given a set of bilinear maps $A\subseteq(\mathbb{N}^{n})^{\mathbb{N}^{n\times n}}$, a vector $\vect{v}$, a vector norm $||.||$ and $F$ such that:
 \begin{itemize}
  \item $(\vect{v},1)\in F$,
  \item for all $(\vect{u},i),(\vect{w},j)\in F$ and $\Phi\in A$, $(\Phi(\vect{u},\vect{w}),i+j)\in F$.
 \end{itemize}
  Is this possible to approximate $\lim_{m\rightarrow\infty} \left(\max\left\{||\vect{u}||:(\vect{u},m)\in F \right\}\right)^{1/m}$ ?
\end{Problem}

Then there are two directions to try to generalize the technique.
Intuitively, we only need a dynamic programming algorithm that can be described with finitely many states and finitely many (multi)linear maps.
We can ask what are other families of graphs where we can apply a similar technique.
For instance, it is clear that everything works well for bounded tree-width and it seems that it should work for bounded clique-width. 
In fact, for clique-width we would also compute a joint spectral radius. However the dimension of the vectors and 
matrices might be to big to be able to do anything.

The last direction is to compute bounds for other kinds of sets.

\begin{Problem}
Can we apply a similar technique to bounds $k$-minimal (resp. minimal, $k$-maximal, maximal) $(\sigma,\rho)$-dominating sets when $\sigma$ and $\rho$ are recognizable?

What if the conditions on $\sigma$ and $\rho$ is weaker?
\end{Problem}
It seems to be doable for $k$-minimal and $k$-maximal, but is not obvious for minimal and maximal in general.
However, it might be the case that for any recognizable sets $\sigma$ and $\rho$, there exists a computable $k$ such that
$k$-minimal  $(\sigma,\rho)$-dominating sets are exactly minimal  $(\sigma,\rho)$-dominating sets.

There is no reason to be restricted to $(\sigma,\rho)$-dominating sets.
It is easy to generalize our method to more than 2 sets with constraints on the size of the intersections of neighborhoods with other sets as long as the allowed degrees are still recognizable sets.
For instance, using 5 different classes one could bound the number of induced collections of $C_4$.
However, a better approach would be to show that this is doable for any property which is expressible in some given logic.

\section*{Acknowledgement}
The author would like to thanks Michaël Rao and G{\"u}nter Rote for discussions or comments on this result.
This work is supported by the French National Research Agency, ANR project GraphEn (ANR-15-CE40-0009).

\bibliographystyle{plain}
\bibliography{biblio}

\appendix

\section{Annex: C++ code}\label{C++desc}
We give a short description of the C++ code used to generate the sets of operators and find the sets $X$ by applying Algorithm \ref{algopw} or Algorithm \ref{findconvsettree}.
First, remark that this code requires the library gmp (for efficient computation on rational numbers). 
This library is easy to install on any Linux (and is often natively installed) and there are many easy solutions for windows (Cygwin and MinGW for instance).

This code is written in different files.
First there is the file \emph{main.cpp}, that reads the options, and calls the different functions with the right values for $\sigma$ and $\rho$ and the class of graphs.

Then two files are dedicated to the computation of the set of operators:
\begin{description}
 \item[automaton.hpp:] This file contains the description of the automatons for the state of a vertex (there are 3 different automatons for enumerating all/minimal/maximal $(\sigma,\rho)$-dominating sets). 
 The automaton is an object that depends on $\sigma$ and $\rho$ (given on a format that corresponds to the automatons recognizing these sets).
 These objects contain the methods that given the state of a vertex decides the new state depending on how its neighborhood changes.
 There are also two special automatons: one for maximal induced matchings and the other for minimal dominating sets in pathwidth $2$.
 \item[get\_operators.hpp:] This file describes three classes that use the automatons from automaton.hpp to compute the operators for trees, forests, pathwidth 1 and pathwidth 2.
 It also contains the functions that call the uses the right class and return the operators and the initial vector.
 There is also a special class for the computation of minimal dominating sets in pathwidth $2$.
\end{description}

Two other files are dedicated to the computation of the set $X$:
\begin{description}
 \item[algebraic.hpp:] This file implements the class \emph{Number} that allow us to do exact computation on algebraic numbers. 
 This is standard and relies on the bijection between the smallest subfield of $\mathbb{R}$ that contains $\mathbb{Q}$ and $\alpha$ and $\mathbb{Q}[X]/P(X)$ where 
 $\alpha$ is an algebraic number of minimal polynomial $P(X)$. Interval of rationals (that can be computed with arbitrary precision) are used to solve inequalities. 
 The only non-trivial operation is the division, but this can be done by using the extended Euclidean algorithm to compute Bézout coefficients and obtain the inverse of a polynomial in $\mathbb{Q}[X]/P(X)$.
 \item[X\_from\_operators.hpp:] This file implements Algorithm \ref{algopw} and Algorithm \ref{findconvsettree}. We use a simple implementation of the simplex algorithm to find the set $\operatorname{conv_<}(S)$.
\end{description}

When using the program one should specify in the arguments which $(\sigma,\rho)$-dominating sets are to be counted in which graph class.
Then before doing anything the program invites the user to give the growth rate to test (an algebraic number) and a possibly empty set of vectors to add to the set $X$. 

For more details on the implementation, one should look at the code. 
For more details on how to use it, one should call the program with the argument -h.

Remark that most of the computation time is spent inside the function that computes $\operatorname{Hull}_<$. 
A better implementation of the simplex or another algorithm computing the $\operatorname{Hull}_<$ could greatly reduce the execution time.
However, we keep the code as simple as possible on this side to improve checkability.

\section[Annex: X of Proposition \ref{treetpd}]{Annex: $X$ of Proposition \ref{treetpd}}\label{Xperfecttotaldomtree}

We give the set $X$ mentioned in the proof of Proposition \ref{treetpd}:

\begin{align*} 
X=&\left\{
\begin{pmatrix} 
0\\   0 \\  2/7\\   1   
\end{pmatrix},\right.
\begin{pmatrix}
0\\   0  \\ \alpha^{6}/14  \\ 13\alpha^{6 }/56 
\end{pmatrix},
\begin{pmatrix} 
0  \\ 0   \\ \alpha^{12 } /56  \\ 3\alpha^{12 } /56 
\end{pmatrix},
\begin{pmatrix}
0 \\  0   \\ \alpha^{18 } /224  \\ 11\alpha^{18 } /896
\end{pmatrix},
\begin{pmatrix} 
0 \\  0   \\ \alpha^{24 }/896   \\ 5\alpha^{24 }/1792
\end{pmatrix},\\
&\begin{pmatrix}  
0 \\  0   \\ \alpha^{30 }/3584   \\ 9\alpha^{30 }/14336
\end{pmatrix},
\begin{pmatrix}  
0 \\  0   \\ \alpha^{36 }  /14336  \\ \alpha^{36 }  /7168
\end{pmatrix},
\begin{pmatrix}
0 \\  0   \\ \alpha^{42 } /57344   \\ \alpha^{42 } /32768
\end{pmatrix},
\begin{pmatrix}
0 \\  0   \\ \alpha^{48 }  /229376   \\ 3\alpha^{48 }  /458752
\end{pmatrix},\\
&
\begin{pmatrix}
0 \\  0   \\ \alpha^{54} /917504  \\ 5\alpha^{54} /3670016
\end{pmatrix},
\begin{pmatrix}
0 \\  0   \\ \alpha^{60 }  /3670016   \\ \alpha^{60 }  /3670016
\end{pmatrix},
\begin{pmatrix}
0 \\  0   \\ \alpha^{66 } /14680064  \\ 3\alpha^{66 } /58720256
\end{pmatrix},\\
&
\begin{pmatrix}
0 \\  0   \\ \alpha^{72 }/58720256   \\ \alpha^{72 }/117440512  
\end{pmatrix},
\begin{pmatrix}
0 \\  0   \\ \alpha^{78 }/234881024   \\ \alpha^{78 }/939524096\
\end{pmatrix},
\begin{pmatrix} 
0   \\ \alpha^{79  } /939524096 \\  0   \\ \alpha^{79  } /234881024 
\end{pmatrix},\\
&
\begin{pmatrix}
0   \\ \alpha^{80} /939524096 \\  0   \\ 3\alpha^{80} /939524096\  
\end{pmatrix},
\begin{pmatrix} 
 \alpha^{80} /939524096  \\ \alpha^{80} /234881024 \\  0 \\  0   
\end{pmatrix},
\begin{pmatrix}
 \alpha^{81 } /939524096  \\ 3\alpha^{81 } /939524096  \\  0  \\ 0  
\end{pmatrix},\\
&
\begin{pmatrix} 
\alpha^{82} /939524096    \\ \alpha^{82} /469762048  \\ 0 \\  0  
\end{pmatrix},
\begin{pmatrix} 
 \alpha^{83 } /939524096  \\ \alpha^{83 } /939524096 \\ 0  \\ 0   
\end{pmatrix},\left.
\begin{pmatrix}
 \alpha^{84} /939524096  \\  0   \\ \alpha^{84} /939524096 \\  0  
\end{pmatrix}\right\}
\end{align*}

\end{document}